\documentclass[
11pt,  reqno]{amsart}
\usepackage[margin=1.1in,marginparwidth=1.5cm, marginparsep=0.5cm]{geometry}

\usepackage{booktabs} 
\usepackage{microtype}
\usepackage{amssymb}
\usepackage{mathrsfs}

\usepackage{upgreek} 

\usepackage{color}
\usepackage[implicit=true]{hyperref}

\usepackage{xsavebox}

\usepackage{cases}

\allowdisplaybreaks[2]

\newcommand{\mathd}{\mathrm{d}}

\definecolor{gr}{rgb}   {0.,   0.69,   0.23 }
\definecolor{bl}{rgb}   {0.,   0.5,   1. }
\definecolor{mg}{rgb}   {0.85,  0.,    0.85}
\definecolor{yl}{rgb}   {0.8,  0.7,   0.}
\definecolor{or}{rgb}  {0.7,0.2,0.2}

\usepackage{tikz} 

\usepackage{marginnote}
\usepackage{scalerel} 

\usetikzlibrary{shapes.misc}
\usetikzlibrary{shapes.symbols}
\usetikzlibrary{decorations}
\usetikzlibrary{decorations.markings}

\newtheorem{theorem}{Theorem} [section]

\newtheorem{lemma}[theorem]{Lemma}

\newtheorem{remark}[theorem]{Remark}

\newtheorem{definition}[theorem]{Definition}
\newtheorem{corollary}[theorem]{Corollary}

\newcommand{\noi}{\noindent}

\newcommand{\R}{\mathbb{R}}
\newcommand{\T}{\mathbb{T}}

\let\Re=\undefined\DeclareMathOperator*{\Re}{Re}

\let\P= \undefined
\newcommand{\P}{\mathbf{P}}

\newcommand{\E}{\mathbb{E}}

\renewcommand{\L}{\mathcal{L}}

\newcommand{\F}{\mathcal{F}}

\newcommand{\al}{\alpha}

\newcommand{\dl}{\delta}

\newcommand{\eps}{\varepsilon}

\newcommand{\ld}{\lambda}

\newcommand{\s}{\sigma}

\newcommand{\ft}{\widehat}

\newcommand{\wt}{\widetilde}
\newcommand{\cj}{\overline}

\newcommand{\dt}{\partial_t}

\DeclareMathSymbol{\wcol}{\mathord}{operators}{"3A}
\newcommand{\wick}[1]{\wcol#1\wcol}

\renewcommand{\l}{\ell}
\renewcommand{\o}{\omega}
\renewcommand{\O}{\Omega}

\newcommand{\les}{\lesssim}
\newcommand{\ges}{\gtrsim}

\newcommand{\jb}[1]
{\langle #1 \rangle}

\newcommand{\ind}{\mathbf 1}

\renewcommand{\S}{\mathcal{S}}

\newcommand{\N}{\mathbb{N}}

\renewcommand{\H}{\mathcal{H}}

\newtheorem*{ackno}{Acknowledgements}

\numberwithin{equation}{section}
\numberwithin{theorem}{section}

\newcommand{\PP}{\mathbb{P}}

\DeclareMathOperator{\Law}{Law}

\newcommand{\W}{\mathcal{W}}

\newcommand{\dr}{\theta}
\newcommand{\Dr}{\Theta}

\usepackage{comment}

\begin{document}
\baselineskip = 14pt

\title[Focusing Gibbs measure]
{Focusing Gibbs measures  with harmonic potential}

\author[T.~Robert, K.~Seong, L.~Tolomeo, Y.~Wang]
{Tristan Robert, Kihoon Seong, Leonardo Tolomeo, Yuzhao Wang}

\address{Tristan Robert\\
IECL -- Facult\'e des sciences et Technologies, 54506 Vandoeuvre-l\`es-Nancy, France.  }

\email{tristan.robert@univ-lorraine.fr}

\address{
Kihoon Seong\\
Max Planck Institute 
for Mathematics In The Sciences\\ 
04103 Leipzig, Germany
}

\email{kihoon.seong@mis.mpg.de}

\address{
Leonardo Tolomeo\\ 
School of Mathematics,
The University of Edinburgh,
James Clerk Maxwell Building, Room 5601,
The King's Buildings, Peter Guthrie Tait Road,
Edinburgh, EH9 3FD, United Kingdom}

\email{l.tolomeo@ed.ac.uk}

\address{
Yuzhao Wang\\
School of Mathematics\\
Watson Building\\
University of Birmingham\\
Edgbaston\\
Birmingham\\
B15 2TT\\ United Kingdom}

\email{y.wang.14@bham.ac.uk}



\begin{abstract}
In this paper,
we study the Gibbs measures associated to the focusing nonlinear Schr\"odinger equation with harmonic potential on Euclidean spaces.
We establish a dichotomy for normalizability vs non-normalizability in the one dimensional case, and under radial assumption in the higher dimensional cases.
In particular,
we complete the programs of constructing Gibbs measures in the presence of a harmonic potential initiated by Burq-Thomann-Tzvetkov (2005) in dimension one and Deng (2013) in dimension two with radial assumption.
\end{abstract}

\date{\today}
\maketitle

\vspace{-5mm}


\section{Introduction}
\label{SEC:1}

In this paper,
we investigate the issue of constructing the Gibbs measure associated with the nonlinear Schr\"odinger equation (NLS) with harmonic potential on Euclidean spaces. This measure is a probability measure $\rho$ on $\mathcal \S'(\R^d)$, the space of Schwartz distributions on $\R^d$, formally written as
\begin{align}
    d\rho (u) = \mathcal Z^{-1} e^{-H(u)}du.
\label{gibbs}
\end{align}

\noi
Here $\mathcal Z$ is a normalizing constant (the partition function), and $H$ is the Hamiltonian of
\begin{align}\label{HNLS}
i\dt u = (-\Delta +|x|^2)u +\sigma|u|^{p-2}u,~~(t,x)\in\R^{1+d},
\end{align}
where $p>2$ is a real number, and $\sigma\in\{-1,1\}$ encodes the nature of the nonlinearity (focusing when $\sigma=-1$ and defocusing when $\sigma=1$). Namely, $H$ is given by
\begin{align}\label{Hamiltonian}
H(u) = \frac12 \int_{\R^d}|\nabla u|^2dx +\frac12\int_{\R^d}|x|^2|u|^2dx +\frac{\sigma}{p}\int_{\R^d}|u|^pdx.
\end{align}

Motivated by the statistical point of view on the nonlinear Schr\"odinger dynamics, the construction and invariance of Gibbs measures of the type \eqref{gibbs} has received a lot of attention in recent years, and we refer to \cite{BO94,BO96,BTTz,CFL,Deng12,DNY2,DNY3,DNY4,LRS,LMW,LW,TW,Tzv1,Tzv2} and references therein for a further account on the subject. In turn, this statistical point of view on nonlinear Schr\"odinger equations, which dates back to the seminal papers \cite{LRS} by Lebowitz, Rose, and Speer, and \cite{BO94} by Bourgain, was inspired by the developments during the 70's and 80's of the probabilistic approach to Euclidean quantum field theory (EQFT). The latter approach proved indeed very powerful to construct relativistic quantum field theories in the Minkowski space-time $\R^{1+(d-1)}$ satisfying the so-called Wightman axioms; again we point to \cite{BG,BG2,GH,GJ,Simon} for more detailed discussions on the matter. One of the main strategy to define an EQFT on $\R^d$ has been to regularize both small scales (``ultraviolet cut-off'') and large scales (``infrared cut-off'') by considering the Gibbs type measures \eqref{gibbs} with a ``truncated'' Hamiltonian, considering for example either a lattice approximation (small scale truncation in physical space), or a regularization by a convolution with an approximation of the identity (small scale truncation in frequency space), and with $\R^d$ replaced by a compact domain (large scale truncation). Then one first proves convergence (up to adding diverging renormalization constants if necessary) of theses approximate measures when removing the small scale truncation on the Hamiltonian, and then convergence from the finite volume case to the infinite volume measure. In this regard, one of the motivations of studying the construction of the Gibbs measure \eqref{gibbs} associated to the Hamiltonian \eqref{Hamiltonian} is that, thanks to the confining potential $|x|^2$, we can directly construct the infinite volume measure without the need for an infrared cut-off.

The first guess to try and give a meaning to the formal expression \eqref{gibbs} consists in interpreting the measure $e^{-H(u)}du$ as the measure
\begin{align}\label{gibbs2}
d\rho(u) = \mathcal Z^{-1}e^{-\frac{\sigma}p\int_{\R^d}|u|^pdx}d\mu(u),
\end{align}
where $\mu$ is the Gaussian measure formally given by
\begin{align}\label{mu1}
d\mu(u)=\wt {\mathcal Z}^{-1}e^{-\frac12\int_{\R^d}\big(|\nabla u|^2 +|x|^2|u|^2\big)dx}du.
\end{align}
One can then interpret $\mu$ as a Gaussian free field on $\R^d$, with the Green function of the harmonic oscillator in place of the usual Green function of the Laplacian as correlation kernel; see \eqref{Gaussian} below for a proper definition of this measure. One then hopes to show that the expression \eqref{gibbs2} defines a probability measure with density with respect to the Gaussian measure $\mu$ in \eqref{mu1}, for some appropriate range of $p$, depending on the sign of $\sigma$.
%
%
%

To set our expectations about the constructibility of the measure \eqref{gibbs2}, let us first discuss the construction of the Gibbs measure in the case without harmonic potential and with infrared cut-off, namely on the $d$-dimensional torus $\T^d$ instead of the Euclidean space $\R^d$:
\begin{align}\label{gibbs3}
d\rho(u) = \mathcal Z^{-1}e^{-\frac{\sigma}{p}\int_{\T^d}|u|^pdx}d\mu(u)
\end{align}
with\footnote{Note that we added the mass $\|u\|_{L^2}^2$ to the quadratic part of the Hamiltonian in the definition of the Gaussian measure $\mu$ in \eqref{gibbs3}. This is to avoid a lack of integrability of the density coming from the zero frequency of $u$, making the following discussion substantially simpler.}
\begin{align*}
d\mu(u)= \mathcal  Z^{-1}e^{-\frac12\int_{\T^d}(|\nabla u|^2+|u|^2)dx}du.
\end{align*} 
First, in the defocusing case $\sigma=1$, as pointed out in \cite{BO94}, there is no issue in dimension $d=1$ in making sense of $\rho$ in \eqref{gibbs3}. Indeed, in this case one has that $$V(u)=\frac1p\int_\T |u|^pdx<+\infty,$$ $\mu$-a.s. for any $p\ge 2$.\footnote{Actually for any $p>0$, but recall from \eqref{HNLS} that we limit our discussion to the case $p\ge 2$.} 
Since this potential energy is bounded from below, one directly concludes that $e^{-\sigma V(u)}\in L^1(\mu)$ for any $p\ge 2$ in the defocusing case $\sigma=1$, allowing indeed to give a meaning to $\rho$ as a measure with density with respect to the Gaussian measure $\mu$. 

However, as soon as the dimension $d\ge 2$, one faces the issue that $V(u) =+\infty$ $\mu$-a.s.\ for any $p\ge 2$. 
This is the starting point of the ultraviolet cut-off and Wick renormalization. Namely, one needs to replace the infinite potential energy $V(u)$ with a truncated renormalized energy $\wick{V_N}(u)=\int_{\T^d}F_{p,N}(u)dx$, where $N\in\N$ is a regularization parameter, and $F_{p,N}$ is both a truncation and a renormalization of $|\cdot|^p$, and subsequently take a limit as $N \to \infty$. 
For example, when $d=2$ and $p=2k$, $F_{p,N}(u)=(-1)^kk!L_k\big(|\Pi_{\le N}u|^2;c_N\big)$. Here $L_k$ denotes the $k$-th Laguerre polynomial, $\Pi_{\le N}$ is the projection of $u$ on Fourier frequencies smaller than $N$, and $c_N =\E\int_{\T^d}|\Pi_{\le N}u|^2dx\underset{N\to\infty}{\longrightarrow}+\infty$ is a renormalization constant. 
Then one can show that $\wick{V(u)}=\lim_{N\to\infty}\wick{V_N(u)}$ is well defined $\mu$-a.s., and that $e^{-\sigma\,:V(u):\,}\in L^1(\mu)$, allowing to make sense of $\rho$. 
See for example \cite{OT} for the case $d=2$ and all even integers $p$. In particular, note that contrarily to the case $d=1$, in the case $d=2$, the renormalized potential $\wick{V(u)}$ is not bounded from below. Therefore, one has to give extra arguments to get the exponential integrability of $-\sigma\wick{V(u)}$, see also \cite{Nelson,Simon,GJ}.
The case $d=3$ is even more singular, and the Gibbs measure \eqref{gibbs3} can only been constructed in the case $p=4$.\footnote{The Gibbs measure has been constructed only for real-valued $u$, contrary to the two-dimensional case where both real-valued \cite{Nelson,Simon,GJ} and complex-valued \cite{OT} $u$ are dealt with. Impossibility of the construction for $p\ge6$ is expected in view of the result \cite{ADC}, but as far as the authors aware, a rigourous proof is not available yet.}  
In this case the construction of the corresponding Gibbs measure requires a further renormalization procedure to get a convergent sequence of approximate measures. Moreover, in this case the limiting measure is singular with respect to $\mu$, see \cite{BG,BG2,GH,GJ}. 
In the higher dimensional case $d\ge 4$, the support of $\mu$ becomes too singular, eventually leading to the Gibbs measure $\rho$ being trivial (i.e.\ Gaussian) in this case \cite{ADC,Fro}, even for the smallest possible value $p=4$.

In the focusing case $\sigma=-1$, the situation changes radically. Indeed, even in the case $d=1$, it was observed by Lebowitz, Rose, and Speer \cite{LRS} that as it is, the Gibbs measure \eqref{gibbs3} is ill-defined for any $p>2$. Indeed the potential energy $-\sigma V(u)$ is not exponentially integrable with respect to $\mu$ in this case. 
Notice that this is not an infinite dimensional phenomenon. Indeed, one can easily check that the measure on $\R$ 
$$ \exp\Big(\frac 1 p |x|^p - \frac 12 |x|^2\Big) dx $$
is not finite for $p>2$, hence it cannot be normalised to a probability measure.
To get around this issue, it was proposed in \cite{LRS} to add a mass cut-off, namely to replace the density $e^{-\sigma V(u)}$ for $\rho$ in \eqref{gibbs3} with 
\begin{align}\label{gibbs4}
d\rho(u)= \mathcal  Z^{-1}\ind_{\{\int_\T|u|^2 dx  \le K\}}\exp\big(-\sigma V(u)\big)d\mu(u),
\end{align} 
for some parameter $K>0$, in order to try and recover the integrability of the density with respect to $\mu$. 
The introduction of the mass cutoff is justified by the fact that the mass is a conserved quantity for the flow of the periodic NLS, hence \eqref{gibbs4} represents a (generalized) grand-canonical ensemble for this equation. 

It was then shown in \cite{LRS} that the measure is indeed normalisable for any $2<p<6$ and any $K>0$, and also for $p=6$ and $K< \|Q\|_{L^2}^2$, where $Q$ is the ground state\footnote{Namely, the (unique up to symmetries) minimizer of the Gagliardo-Nirenberg-Sobolev inequality on $\R$ with $\|Q\|_{L^6(\R)}^6 = 3 \| Q'\|_{L^2(\R)}$.} for the quintic NLS on $\R$, while in the case $p=6$ and $K>\|Q\|_{L^2}^2$ or $p>6$ and any $K>0$, the measure \eqref{gibbs4} also becomes ill-defined. Subsequently in \cite{BO94}, another proof of this result was given by Bourgain, that corrected an issue in the original paper \cite{LRS}. More recently, the third author, together with Oh and Sosoe, showed in \cite{OST} that the measure remains well-defined up to the end-point for the mass threshold $K=\|Q\|_{L^2}^2$. This shows a \textit{phase transition} for the one-dimensional focusing Gibbs measure at criticality: the partition function $\mathcal Z_K=\int \mathbf{1}_{\{\|u\|_{L^2}\le K\}} e^{-\sigma V(u)}d\mu(u)$ is not analytic with respect to $K$ when $p=6$. Similar phase transitions at critical exponents between strongly/weakly nonlinear regimes have been observed in other focusing models, see \cite{OOT1,OOT2}. The situation is even more dramatic in higher dimension: in $d=2$, Brydges and Slade showed in \cite{BS} that even for $p=4$ and with the same Wick renormalization as in the defocusing case, the measure is still ill-defined. See also \cite{BB,LW,OST,R,Tzv1,Tzv2,Xian} for generalizations of these results to $d=2$ on the unit disc or $d\ge 1$ with fractional dispersion.

Notice that here we only discussed the construction of the Gibbs measure with infrared cut-off; however removing this cutoff in order to study the measure on the euclidean space is a highly non-trivial endeavour. See for example the recent works \cite{BG3,GH,OTWZ, BL}.

As mentioned before, besides the physical relevance of \eqref{HNLS} to describe Bose-Einstein condensates \cite{BBDBG,DS}, one of the interests in adding the harmonic potential to the Hamiltonian \eqref{Hamiltonian} is that it allows to work directly without infrared cut-off in constructing the Gibbs measure \eqref{gibbs}. However, compared to the previous discussion, much less is known regarding the construction of the measure \eqref{gibbs2} in this case.

The issue of constructing the Gibbs measure \eqref{gibbs2} was initiated by Burq-Thomann-Tzvetkov \cite{BTTz} in dimension $d = 1$. As in the case of the defocusing Gibbs measure without harmonic potential on the circle \eqref{gibbs3}, there is no issue in making sense of \eqref{gibbs2} in the defocusing case $\sigma=1$ for any $p>2$ since it still holds $\int_\R |u|^pdx<+\infty$ $\mu$-a.s. (see for example Corollary~\ref{COR:intp}~(i) below), and since the potential energy $-\frac{\sigma}{p}\int_{\R^d}|u|^p dx$ is bounded from above in this case, thus exponentially integrable with respect to $\mu$. However, one faces the same issue as for \eqref{gibbs3} in the focusing case $\sigma=-1$: the potential energy is no longer bounded from above, and in particular one can show that it is not exponentially integrable anymore. Thus, as in \eqref{gibbs4}, one needs to consider instead the measure with a mass cut-off in the density
\begin{align}\label{Gibbs2}
    d\rho (u) = \mathcal Z^{-1} \ind_{\{|\int_\R :|u|^2: dx | \le K\}} \exp \bigg(-\frac{\sigma}p \int_{\R^d} |u|^p dx \bigg) d\mu (u),
\end{align}
with the Gaussian measure $\mu$ as in \eqref{mu1}. However, an important difference with \eqref{gibbs4} is that even for $d=1$, $\int_\R |u|^2dx=+\infty$ $\mu$-a.s. in the case with harmonic potential. This is why the cut-off in the density in \eqref{Gibbs2} is defined with respect to the \textit{Wick-ordered} mass, namely for a renormalized version of $\|u\|_{L^2(\R)}^2$ as discussed above; see \eqref{Wick} below for the proper definition. This is reminiscent of what happens for \eqref{gibbs3} in higher dimensions $d\ge 2$, where the mass, and actually any nonlinear functional of $u$, becomes ill-defined $\mu$-a.s. and needs to be renormalized. Let us stress that in our case, though, only the $L^2$ norm is $\mu$-a.s. infinite and needs to be renormalized, while the potential energy $\int_\R|u|^pdx$ is $\mu$-a.s. finite for any $p>2$; see Corollary~\ref{COR:intp}~(i) below. Thus the model \eqref{gibbs} with  Hamiltonian \eqref{Hamiltonian} in dimension $d=1$ has an interesting behavior in between the $d=1$ and $d=2$ cases for the measure \eqref{gibbs3} in finite volume without harmonic potential. 

Then in the one-dimensional focusing case $d=1,\sigma=-1$, the measure with Wick-ordered mass cut-off \eqref{Gibbs2} was first constructed by Burq, Thomann, and Tzvetkov \cite{BTTz} in the case $p=4$, and they also prove its invariance under \eqref{HNLS}. However, the whole picture was left incomplete as their argument does not work for general $p > 2$.
The main purpose of this paper is to complete the program initiated in \cite{BTTz} of constructing the focusing Gibbs measures $\rho$ in \eqref{Gibbs2}. 
In particular, we identify a sharp threshold $p^* = 6$,
such that when $2 < p < p^*$ the focusing Gibbs measure \eqref{Gibbs2} is normalizable, 
while it is not normalizable when $p \ge p^*$.
See Theorem \ref{THM:main} below for a detailed statement.
Note that in the case $p=p^*$ the Gibbs measure \eqref{Gibbs2} is \emph{not} normalizable, as opposed to the situation for \eqref{gibbs3} on the one-dimensional torus. 
We also remark that our argument, which is based on a variational formula, is different from the one in \cite{BTTz}.

In the two-dimensional case $d=2$, the construction of the focusing measure \eqref{Gibbs2} with $\sigma=-1$ was first studied by Deng \cite{Deng12} under a radial assumption. Namely, he replaced the Gaussian measure $\mu$ in \eqref{mu1}, defined on $\S'(\R^d)$, by a Gaussian measure defined on the space of radial Schwartz distributions $\S'_{\textup{rad}} (\R^2)$, and showed that the focusing Gibbs measure (with Wick-ordered $L^2$-cutoff) can be constructed when $2 < p < 4$. He also showed that in the defocusing case $\sigma=1$, the measure \eqref{gibbs2} can be constructed for any even $p\ge 4$, and that in both cases the Gibbs measure is invariant under \eqref{HNLS}.
In this paper, 
we shall complete this line of research by also showing the non-normalizability of the \emph{focusing} Gibbs measure with Wick-ordered $L^2$-cutoff for $p \ge 4$.
It turns out that our argument extends to higher dimensions, still under the radial assumption. In particular,
we show that there is a critical index $p^*(d) = 2 + \frac{4}d$ such that the focusing Gibbs measure with Wick-ordered $L^2$-cutoff is normalizable when $2 < p < p^*(d)$,
but non-normalizable when $p \ge p^*(d)$.\footnote{As a matter of fact, there will be a further restriction $p < \frac{2d}{d-2}$ when $d \ge 3$. See Theorem \ref{THM:main} below for the detailed statements.}

Without the radial assumption, the only result that we are aware of is that of \cite{dBDF2} constructing the measure \eqref{gibbs2} in the two-dimensional cubic defocusing case $d=2,\sigma=1,$ and $p=4$, and proving its invariance under \eqref{HNLS}. In this case one also needs to renormalize the potential energy as in the two-dimensional finite volume case without harmonic potential \eqref{gibbs3}.

To conclude this introduction, let us stress one more time that the one-dimensional focusing Gibbs measures \eqref{Gibbs2} and \eqref{gibbs3} with $d = 1$ share many common features, but that our study
shows that their behaviors at the critical exponent $p=6$ are quite distinct: the former is not normalizable while the latter exhibits phase transition. A similar situation appears in the two-dimensional case with radial assumption. In this case the focusing Gibbs measure defined on radial functions on the unit disc $B_2$ with the Dirichlet boundary condition is formally given by
\begin{align}
d\rho_{2,p} (u) = \mathcal Z_{2,p}^{-1}
\ind_{\{\|u\|_{L^2(B_2)}^2\le K\}}
 \exp \bigg(\frac1{p} \int_{B_2} |u|^p dx \bigg) d\mu_2 (u),
\label{Gibbs5}
\end{align} 

\noi
where $\mu_2$ is the corresponding Gaussian measure,
namely the law of the Dirichlet Gaussian free field on $B_2$.
Bourgain-Bulut \cite{BB},
Oh-Sosoe-Tolomeo \cite{OST1},
and Xian \cite{Xian}
showed that there is a phase transition at the critical nonlinearity $p = 4$.
In particular,
they constructed the measure \eqref{Gibbs5} provided (i) $2 < p < 4$
and (ii) $p=4$ and $0 < K \le \|Q\|_{L^2(B_2)}^2$, where $Q$ is the minimizer of the Galiardo-Nirenberg-Sobolev inequality on $\R^2$,
while \eqref{Gibbs5} is not normalizable for (iii) $p = 4$ and $K > \|Q\|_{L^2(B_2)}^2$ and (iv) $p > 4$.
As a sharp contrast,
we will show that there is no phase transition occurring at the critical nonlinearity $p=4$ for \eqref{Gibbs2} with $d = 2$ defined on radial functions.

\subsection{Harmonic potential}
To state more precisely our results, we start with a discussion on the one dimensional Laplacian with harmonic potential on $\R$: 
\begin{align} 
\L = -\partial_x^2 + x^2,
\label{HP}
\end{align}

\noi
associated with the following positive quadratic form $Q$
defined on $C_c^\infty (\R)$ by
\[
Q(f) : = \int_\R |f' (x)|^2 + x^2 |f(x)|^2 dx.
\]

\noi
Let us collect some facts related to the operator $\L$.
The operator $\L$ has a positive self-adjoint extension on $L^2(\R)$
and has eigenfunctions $\{h_n\}_{n \ge 0}$ with
\begin{align} 
h_n (x) = (-1)^n c_n e^{\frac{x^2}2} \frac{d^n}{dx^n} (e^{-x^2})
\label{eigenf}
\end{align}

\noi
and $c_n = (n!)^{-\frac12} 2^{-\frac{n}2} \pi^{-\frac14}$.
Then $\{h_n\}_{n \ge 0}$ is a complete normal basis of $L^2(\R)$.
Let $\ld_n^2$ be the corresponding eigenvalues, i.e.
$\L h_n = \lambda_n^2 h_n$.
Then it is known (see for example \cite{BTTz}) that 
\begin{align}
\lambda_n =  \sqrt{1+2n}.
\label{ld}
\end{align} 

\noi
We will need the following estimates on the eigenfunctions $h_n$ from \cite{YZ}.
\begin{lemma}
\label{LEM:eig1}
With $h_n$ defined in \eqref{eigenf}, 
we have the following estimate
\begin{align}\label{hn}
\| h_n(x) \|_{L^p(\R)}
\lesssim \begin{cases}
\lambda_n^{-\frac13 + \frac2{3p}} & \textup{if } 2 \le p \le 4\\
\lambda_n^{-\frac16} &  \textup{if }  p \ge 4.
\end{cases},
\end{align}
uniformly in $n\in\N$, $p\ge 2$.
\end{lemma}

We then define the Sobolev spaces associated to the operator $\L$.

\begin{definition}
\label{DEF:sob}
For $1\le p \le \infty$
and $s \in \R$,
we define the harmonic Sobolev space $\W^{s,p} (\R)$
by the norm
\[
\| u\|_{\W^{s,p} (\R)} = \| \L^{\frac{s}2} u\|_{L^p (\R)}.
\]

\noi
When $p=2$,
we write $\W^{s,2} (\R) = \H^s (\R)$
and for $u = \sum_{n=0}^\infty c_n h_n$ we have
$\|u \|_{\H^s (\R)}^2 = \sum_{n=0}^\infty \ld^{2s}_n|c_n|^2$.
\end{definition}

From the above definition, 
we see that $\|f \|_{\H^0 (\R)} = \|f \|_{L^2(\R)}$ and
\begin{align}
\label{H1}
\|f \|_{\H^1 (\R)}^2  = Q(f) =  \int_\R |f' (x)|^2 + x^2 |f(x)|^2 dx.
\end{align}

\noi
We recall the following characterization of harmonic Sobolev spaces.

\begin{lemma}[\cite{DG09}]
\label{LEM:sobolev}
For any $1 < p < \infty$, $s \ge 0$,
there exists $C>0$ such that
\begin{align}
\frac1C \|u\|_{\W^{s,p} (\R)} \le \| \jb{D_x}^s u \|_{L^p (\R)} 
+  \| \jb{x}^s u \|_{L^p (\R)} \le C \|u\|_{\W^{s,p} (\R)}
\label{sobolev_1}
\end{align}
\end{lemma}

We also recall the Gagliardo-Nirenberg-Sobolev inequality on $\R$,
\begin{equation}
\label{GNSdisp}
\begin{split}
\|u\|_{L^{p}(\R)}^p & \les  \| u\|^{\frac{p-2}{2}}_{\dot H^1 (\R)}\|u\|^{1+ \frac{p}2}_{L^2(\R)},
\end{split}
\end{equation}

\noi
for $p > 2$, which together with \eqref{H1} implies 
\begin{equation}
\label{GNSdisp1}
\begin{split}
\|u\|_{L^{p}(\R)}^p & \les  \| u\|^{\frac{p-2}{2}}_{\H^1 (\R)}\|u\|^{1+ \frac{p}2}_{L^2(\R)},
\end{split}
\end{equation}

\noi
for $p > 2$.

\subsection{Measure construction and main results}
In this subsection,
we describe the procedure we use to construct the Gibbs measures \eqref{gibbs} and \eqref{Gibbs2} with $d = 1$,
and make precise statements of our main results for the one dimensional case.
We postpone the higher dimensional radial cases to the next subsection. 

With the above notations, 
the Hamiltonian \eqref{Hamiltonian} can be rewritten as
\begin{align}
\begin{split}
H(u) & = \frac12 \int_\R |\L^{\frac12} u(x) |^2 dx - \frac1p \int_\R |u(x)|^p dx 
\end{split}
\label{Hamil}
\end{align}

\noi
In particular, using the eigenbasis $\{h_n\}_{n\ge 0}$ of the previous subsection, we can decompose any $u\in\S'(\R)$ as 
$$u = \sum_{n= 0}^\infty u_n h_n,~~~~~~~~u_n=\langle u,h_n\rangle_{L^2(\R)}.$$

\noi
Then, under the coordinates $u = (u_n)$, the Hamiltonian \eqref{Hamil} has the form
\[
H(u) = H \Big(\sum_{n=0}^\infty u_n h_n \Big) = \frac12 \sum_{n=0}^\infty \lambda_n^2 |u_n|^2
- \frac1p \int_\R \Big| \sum_{n=0}^\infty u_n h_n (x) \Big|^p dx.
\]

\noi
From the above computation,
we may define the Gaussian measure with the Cameron-Martin space $\mathcal H^1(\R)$ formally given by
\begin{align}
    d\mu = \mathcal Z^{-1} e^{-\frac12 \|u\|_{\mathcal H^1}^2} du = \mathcal Z^{-1} \prod_{n=0}^\infty e^{-\frac12 \ld^2_n |u_n|^2} d u_n,
    \label{Gaussian}
\end{align}

\noi
where $du_n$ is the Lebesgue measure on $\mathbb C$.
We note that this Gaussian measure $\mu$ is the induced probability measure under the map
\begin{align}
    \label{maps}
    \o \in \O \longmapsto u^\o = \sum_{n\ge 0} \frac{g_n (\o)}{\ld_n} h_n,
\end{align}

\noi
where $\{g_n\}_{n \in \mathbb N}$ is a sequence of independent standard complex-valued Gaussian random variables on a probability space $(\O, \F, \PP)$.
From \eqref{ld} we see that
\begin{align}
\E \big[\| u^\o\|^2_{L^2 (\R)} \big] = \sum_{n=0}^\infty \frac1{\ld_n^2}  = \infty ,
\label{L2}
\end{align}

\noi
which implies that a typical function $u$ in the support of $\mu$ is not square integrable, and thus a renormalization on the $L^2$-norm $\int_\R |u|^2 dx$ is needed.
On the other hand, a computation similar to that in the proof of Corollary \ref{COR:intp} (i)
yields that
$\E [\|u^\o\|_{L^p (\R)}] < \infty$ when $p > 2$.
See also \cite{IRT16}.
Therefore, the potential energy $\frac1{p} \int_\R |u|^p dx$ does not require a renormalization.

To define the Gaussian measure $\mu$ in \eqref{Gaussian} rigorously, 
we start with a finite dimensional version.
First define the spectral projector $\P_N$ by
\[
\P_N u = \P_N \Big( \sum_{n=0}^\infty u_n h_n \Big) = \sum_{n=0}^N u_n  h_n,
\]

\noi
whose image is the finite dimensional space
\[
E_N = \textup{span} \{h_0, h_1, \cdots, h_N\}.
\]

\noi
Through the isometric map 
\begin{align}
(u_n)_{n=0}^N \mapsto \sum_{n=0}^N u_n h_n,
\label{iso}
\end{align}

\noi
from $\mathbb C^{ N+1}$ to $E_N$, 
we may identify $E_N$ with $\mathbb C^{N+1}$.
Consider a Gaussian measure on 
$\mathbb C^{N+1}$ (or on $ \R^{2N+2}$) given by
\[
d  \mu_N = \prod_{n=0}^N \frac{\ld_n^2}{2\pi} e^{-\frac{\ld_n^2}2 |u_n|^2} d u_n d \overline{u}_n,
\]

\noi
where $N \ge 1$.
The Gaussian measure $ \mu_N$ defines a measure on the finite dimensional space $E_N$ 
via the map \eqref{iso},
which will be also denoted by $\mu_N$.
Then, this $\mu_N$ is the induced
probability measure under the map 
\begin{align}
\o \mapsto u_N^\o  :  = \sum_{n=0}^N \frac{g_n (\o)}{\ld_n}  h_n  .
\label{RVN}
\end{align}

\noi
Given any $s >0$, 
the sequence $(u_N^\o)$
is a Cauchy sequence in $L^2 (\O; \H^{- s} (\R))$ converging to $u^\o$ given in \eqref{maps}.
In particular,
the distribution of the random variable $u^\o \in \H^{- s}$ is the
Gaussian measure $\mu$.
The measure $\mu$ can be decomposed as
\begin{align}
\label{mu}
\mu = \mu_N \otimes \mu_N^\perp,
\end{align}

\noi
where the measure $\mu_N^\perp$ is the distribution of the 
random variable given by 
\[
u_N^{\o,\perp} (x):  = \sum_{n = N+1}^\infty \frac{g_n (\o)}{\ld_n} h_n (x).
\]




Recall from the discussion in the introduction that, inspired by Lebowitz-Rose-Speer \cite{LRS},
to define the focusing Gibbs measure \eqref{Gibbs2}, 
a proper mass cut-off is necessary.
But as we pointed out before, from \eqref{L2} we see that $u^\o \notin L^2(\R)$ $\mu$-a.s.,
which motivates us to introduce a Wick-ordered $L^2$-cutoff as in \cite{BO99,BTTz,OST1}.
Given $x \in \R$, $u_N (x)$ in \eqref{RVN} 
is a mean-zero complex-valued Gaussian random variable with variance
\begin{align}
    \label{variance}
    \s_N (x) = \E \big[ |u_N^\o (x)|^2 \big] = \sum_{0 \le n \le N} \frac{h_n^2(x)}{\ld_n^2} ,
\end{align}

\noi
from which we have 
\[
\E\|u_N\|_{L^2}^2=\int_\R \s_N (x) dx = \sum_{n \le N} \frac{1}{\ld_n^2}  \sim \log N \to \infty 
\]

\noi
as $N \to \infty$.
Here $\s_N$ depends on $x\in \R$ as the random process $u^\o$ given by \eqref{maps} is not stationary.
We can then define the Wick power $\wick{|u_N|^2}$ via
\begin{align}
\wick{|u_N|^2} =  |u_N|^2 - \s_N.
\label{Wick}
\end{align}

\noi
It is also known, see for instance \cite[Lemma 3.6]{BTTz}, that $\int_\R :|u_N (x)|^2: dx $ is a Cauchy sequence in $L^2(\H^{-s} (\R), d\mu)$ and converges to a limit, denoted by $\int_\R :|u (x)|^2: dx $, for any $s >0$. 

The main purpose of this paper is then to define the focusing
Gibbs measure \eqref{Gibbs2} with $\sigma=-1$ under the Wick-ordered $L^2$-cutoff $K>0$.
We start with a finite dimensional approximation.
\begin{align}
d \rho_{N} (u) =  \mathcal  Z_{K,N}^{-1} \ind_{ \{ | \int_{\R} : | u_N (x)|^2 : dx | \le K\}}  e^{\frac1p{\| u_N \|_{L^p (\R)}^p}} d\mu_N (u_N)  \otimes d \mu_N^\perp (u_N^\perp),
\label{tru_rho}
\end{align} 

\noi
where $u_N = \P_N u$, $u_N^\perp = u - \P_N u$, and the partition function $Z_{K,N}$ is given by
\begin{align}
    \label{partition}
    \mathcal Z_{K,N} = \int \ind_{ \{ | \int_{\R} : | u_N (x)|^2 : dx | \le K\}}  e^{\frac1p{\| u_N \|_{L^p (\R)}^p}} d\mu (u)  .
\end{align}

\noi
The main result of this paper is to show the sharp criterion so that \eqref{tru_rho} converges to a probability measure as $N \to \infty$.

\medskip

We now state our first main result.

\begin{theorem}
\label{THM:main}
The following statements hold\textup{:}

\smallskip

\noi
\begin{itemize}
\item
[\textup{(i)}] \textup{(subcritical case)}
Let $K > 0$.
If $2< p< 6$, then we have the uniform exponential integrability of the density:
given any finite $r \ge 1$, 
then we have 
\begin{align}
    \label{uniint_p}
    \sup_{N \in \mathbb N} \Big\|  \ind_{ \{ | \int_{\R} : | u_N (x)|^2 : dx | \le K\}}  e^{\frac1 p{\| u_N \|_{L^p (\R)}^p}} \Big\|_{L^r (\mu)}   < \infty .
\end{align}

\noi
Moreover, we have
\begin{align}
    \lim_{N \to \infty}  \ind_{ \{ | \int_{\R} : | u_N (x)|^2 : dx | \le K\}}  e^{\frac1 p{\| u_N \|_{L^p (\R)}^p}} 
    =  \ind_{ \{ | \int_{\R} : | u (x)|^2 : dx | \le K\}}  e^{\frac1 p{\| u \|_{L^p (\R)}^p}} 
    \label{cov-lp}
\end{align}

\noi
in $ L^r(\mu)$.
As a consequence,
the Wick-ordered $L^2$-truncated Gibbs measure $\rho_{N,K}$ in \eqref{tru_rho} converges, in total variation,
to the focusing Gibbs measure $\rho$ defined by
\begin{align}
    \label{rho}
    d \rho (u) = \mathcal  Z_{K}^{-1} \ind_{ \{ | \int_{\R} : | u (x)|^2 : dx | \le K\}}  e^{\frac1 p{\| u \|_{L^p (\R)}^p}} d\mu (u)  .
\end{align}

\noi
Furthermore, the resulting measure $\rho$ is absolutely continuous with respect to the base Gaussian free field $\mu$ in \eqref{mu}.

\smallskip

\noi
\item[\textup{(ii)}] \textup{(critical/supercritical cases)}
Let $p\ge 6$.
Then, for any $K >0$, we have
\begin{align}
    \label{non_int}
    \sup_{N \in \mathbb N} \mathcal Z_{K,N} = \sup_{N \in \mathbb N} \Big\|  \ind_{ \{ | \int_{\R} \wick{| u_N (x)|^2} dx | \le K\}}  e^{\frac1 p{\| u_N \|_{L^p (\R)}^p}} \Big\|_{L^1 (\mu)}   = \infty,
\end{align}

\noi
where $\mathcal Z_{K,N}$ is the partition function given in \eqref{partition}.
The same divergence holds for $\mathcal Z_K$, i.e.\ 
\begin{align}
    \label{non_int2}
   \mathcal Z_{K} =  \int \ind_{ \{ | \int_{\R} \wick{| u (x)|^2} dx | \le K\}}  e^{\frac1 p{\| u \|_{L^p (\R)}^p}}  d\mu   = \infty.
\end{align}
As a consequence, 
the focusing Gibbs measure \eqref{gibbs},
even with a Wick-ordered $L^2$-cutoff,
i.e. \eqref{rho}, 
can not be defined as a probability measure when $p \ge 6$. 

\end{itemize}

\end{theorem}

\medskip

When $p=4$, 
Theorem \ref{THM:main} - (i) provides an alternative proof of the normalizability result for the focusing measure \eqref{rho} by 
Burq-Thomann-Tzvetkov \cite{BTTz}, 
whose proof is based on a Fourier analysis approach developed by Bourgain \cite{BO94}.
Our proof of Theorem \ref{THM:main} - (i), however, is based on the variational approach due to Barashkov-Gubinelli \cite{BG},
which is robust enough to treat all subcritical nonlinearities $2 < p < 6$ simultaneously.   
See Lemma \ref{LEM:var} for the Bou\'e-Dupuis variational formula \cite{BD,Ust}, which plays a key role in this argument.
As a matter of fact,
we will only show the uniform bound \eqref{uniint_p}.
Once the uniform bound \eqref{uniint_p} is established, the $L^r$-convergence \eqref{cov-lp} of the densities follows from the convergence in measure of the densities. See Remark 3.8 in \cite{Tzv2}.  

The proof of Theorem \ref{THM:main} - (ii) is inspired by the recent works by the third author and collaborators \cite{OOT1,OOT2,TW} and by Oh, the second and third authors \cite{OST}, where other non-normalizability results are shown for different measures.
The main step is to construct a drift term which approximates a blow-up profile, such that the Wick-ordered $L^2$-cutoff does not exclude this blow-up profile for any cutoff size $K >0$.
In particular, when $p \ge 6$,
we show that the Wick-ordered $L^2$-cutoff does not exclude the blow-up profiles, which drives the energy functional to blow up. See Section \ref{SEC:non} for more details.

Theorem \ref{THM:main} completes the program initiated by Burq-Thomann-Tzvetkov \cite{BTTz} on constructing the focusing Gibbs measure \eqref{rho}
by providing a sharp criterion.
In particular, we show that the focusing Gibbs measure with a Wick-ordered $L^2$-cutoff is normalizable when $ p< 6$ and non-normalizable when $p\ge 6$.
The main novelty of our result lies in the critical case $p=6$,
where we show the non-normalizability of the focusing Gibbs measure.
Again, we point out that our result provides a sharp contrast with the recent work of Oh-Sosoe-Tolomeo \cite{OST1}, where they show the phase transition occurs at the critical case $p=6$ for the Gibbs measure on $\T$.

\medskip

We shall prove Theorem \ref{THM:main} (i) in Section \ref{SEC:nor},
and (ii) in Section \ref{SEC:non}.
All the proofs are done by exploiting the variational approach by Barashkov-Gubinelli in \cite{BG}.
In the critical/supercritical cases, we show the blow up of the partition function $\mathcal Z_{K,N}$ by constructing a drift term along the blow up profiles.
Similar argument appeared in \cite{LW,OOT1,OST,R}.

\subsection{Higher dimensional cases}
We can extend the main result Theorem \ref{THM:main} for dimensional one to higher dimensions with radial assumption.
In this subsection, we explain the higher dimensional analogue of Theorem \ref{THM:main}.
Recall the Laplacian confined with harmonic potential
\begin{align}
\label{HPd}
\mathcal L = -\Delta + |x|^2,
\end{align}

\noi
where $\Delta= \sum_{i=1}^d \partial_{i}^2$ is the Laplacian operator on $\R^d$ and $|x|^2 = \sum_{i=1}^d x_i^d$ is the harmonic potential.
We define the radial Sobolev spaces $\mathcal W^{s,2}_{\textup{rad}} (\R^d)$ of radial functions in higher dimension induced by the norm:
\[
\| u\|_{\mathcal W^{s,p} (\R^d)} = \| \mathcal L^{\frac{s}2} u \|_{L^p (\R^d)}.
\]

\noi
We also write $\mathcal W^{s,2}_{\textup{rad}} (\R^d) = \mathcal H^s_{\textup{rad}} (\R^d)$.
Similar to Lemma \ref{LEM:sobolev}, we have
\begin{align}
\frac1C \|u\|_{\W^{s,p} (\R^d)} \le \| \jb{\nabla}^s u \|_{L^p (\R^d)} 
+  \| \jb{x}^s u \|_{L^p (\R^d)} \le C \|u\|_{\W^{s,p} (\R^d)}.
\label{sobolev_d}
\end{align}

\noi
Analogously to the one dimensional case \eqref{GNSdisp1},
we have the Gagliardo-Nirenberg-Sobolev inequality on $\R^d$ for $d\ge 1$:
let \textup{(i)} $p > 2$ if $d = 1,2$ and \textup{(ii)} $2 < p < \frac{2d}{d-2}$ if $d \ge 3$, then we have
\begin{equation}
\label{GNSdisp2}
\begin{split}
\|u\|_{L^{p}(\R^d)}^p & \les   \| u\|^{\frac{(p-2)d}{2}}_{\H^1 (\R^d)}\|u\|^{2+\frac{p-2}{2} (2-d)}_{L^2(\R^d)}.
\end{split}
\end{equation}

\noi
It is known \cite{IRT16} that the symmetric operator $\mathcal L$ has a self-adjoint extension on $\mathcal H^1_{\textup{rad}} (\R^d)$ with eigenvalues
\[
\ld_n^2 = 4n + d
\]

\noi
for $n \ge 0$ and associated eigenfunctions $h_n$, i.e. $
\mathcal L h_n = \ld_n^2 h_n$.
Here we abuse notations by using $\ld_n$ and $h_n$ to denote eigenvalues and eigenfunctions of the radial harmonic oscillator \eqref{HPd}. 
We will need the following dispersive estimate on the radial eigenfuncsions $h_n$.
For more details, please see \cite{IRT16} and reference therein.

\begin{lemma}[Proposition 2.4 in \cite{IRT16}]
\label{LEM:eig2}
Let $d \ge 2$. Then we have the following estimate
\begin{align}
    \label{en}
    \|h_n\|_{L^p(\R^d)} \les
    \begin{cases}
    \ld_n^{-d (\frac12 - \frac1p)}, & \textup{ for } p \in [2, \frac{2d}{d-1});\\
    \ld_n^{- \frac12} \log^{\frac1p} \ld_n, & \textup{ for } p = \frac{2d}{d-1} ;\\
    \ld_n^{ d (\frac12 - \frac1p) -1}, & \textup{ for } p \in ( \frac{2d}{d-1}, \infty].
    \end{cases}
\end{align}

\noi
Here the implicit constant only depends on $d$. 
\end{lemma}
Similarly to \eqref{Gaussian},
we may define the Gaussian measure with the Cameron-Martin space $\mathcal H^1_{\textup{rad}} (\R^d)$ formally given by
\begin{align}
    d\mu = \mathcal Z^{-1} e^{-\frac12 \|u\|_{\mathcal H^1}^2} du = \mathcal Z^{-1} \prod_{n=0}^\infty e^{-\frac12 \ld_n^2 |u_n|^2} d u_n d\overline{u}_n,
    \label{Gaussian2}
\end{align}

\noi
where $u_n = \jb{u,h_n}$ and $du_n d\overline{u}_n$ is the Lebesgue measure on $\mathbb C$.
We note that this Gaussian measure $\mu$ is the induced probability measure under the map
\begin{align}
    \label{maps2}
    \o \in \O \longmapsto u^\o = \sum_{n\ge 0} \frac{g_n (\o)}{\ld_n} h_n,
\end{align}

\noi
which converges in $\mathcal H^{-s}_{\textup{rad}} (\R^d)$ almost surely for any $s>0$.
In particular, $u^\o \in \mathcal H^{-s}_{\textup{rad}} (\R^d)$ for any $s > 0$.

Given $x \in \R$ and $u^\o$ in \eqref{maps2}, then $u_N = \P_{\le N} u^\o$ is a mean-zero complex-valued Gaussian random variable with variance
\begin{align}
    \label{variance2}
    \s_N (x) = \E \big[ |u_N^\o (x)|^2 \big] = \sum_{0 \le n \le N} \frac{h_n^2(x)}{\ld_n^2} .
\end{align}

\noi
We define the Wick powers $\wick{|u_N|^2}$ via
\begin{align}
\wick{|u_N|^2} =  |u_N|^2 - \s_N.
\label{Wick2}
\end{align}

\noi
Similarly as in the one-dimensional case,
one can show that $\int_{\R^d} \wick{|u_N (x)|^2} dx $ is a Cauchy sequence in $L^2(\H^{-s}_{\textup{rad}} (\R^d), d\mu)$ and converges to a limit, denoted by $\int_{\R^d} \wick{|u (x)|^2} dx $, for any $s >0$. 
We then define the finite dimensional version of the Gibbs measure of radial functions on $\R^d$,
\begin{align}
d \rho_{N} (u) =  \mathcal  Z_{N}^{-1} \ind_{ \{ | \int_{\R^d} \wick{| u_N (x)|^2} dx | \le K\}}  e^{\frac1p{\| u_N \|_{L^p (\R^d)}^p}} d\mu_N (u_N)  \otimes d \mu_N^\perp (u_N^\perp),
\label{tru_rho2}
\end{align} 

\noi
where $u_N = \P_N u$, $u_N^\perp = u - \P_N u$, and the partition function $\mathcal Z_{N}$ is given by
\begin{align}
    \label{partition2}
    \mathcal Z_{N} = \int \ind_{ \{ | \int_{\R^d} \wick{| u_N (x)|^2} dx | \le K\}}  e^{\frac1p{\| u_N \|_{L^p (\R^d)}^p}} d\mu (u)  .
\end{align}

\noi
As in the 1-dimensional case, we give a sharp criterion so that \eqref{tru_rho} converges to a probability measure as $N \to \infty$.

\medskip

We now state our second main result.

\begin{theorem}
\label{THM:main2}
Given $d \ge 2$, 
define $p^* (d) = 2 + \frac{4}d$.
Then, 
the following statements hold\textup{:}

\smallskip

\noi
\begin{itemize}
\item
[\textup{(i)}] \textup{(subcritical case)}
Let $K > 0$.
Given $2< p< p^*(d)$, for any finite $r \ge 1$, then we have 
\begin{align}
    \label{uniint_d}
    \sup_{N \in \mathbb N} \Big\|  \ind_{ \{ | \int_{\R^d} : | u_N (x)|^2 : dx | \le K\}}  e^{\frac1 p{\| u_N \|_{L^p (\R^d)}^p}} \Big\|_{L^r (\mu)} < \infty .
\end{align}

\noi
Moreover, we have
\begin{align}
    \lim_{N \to \infty}  \bigg(\ind_{ \{ | \int_{\R^d} : | u_N (x)|^2 : dx | \le K\}}  e^{\frac1 p{\| u_N \|_{L^p (\R^d)}^p}} \bigg)
    =  \ind_{ \{ | \int_{\R^d} : | u (x)|^2 : dx | \le K\}}  e^{\frac1 p{\| u \|_{L^p (\R^d)}^p}} 
    \label{cov-lp_d}
\end{align}

\noi
in $ L^r(\mu)$.
As a consequence,
the Wick-ordered $L^2$-truncated Gibbs measure $\rho_{N,K}$ in \eqref{tru_rho2} converges, in total variation,
to the focusing Gibbs measure $\rho$ defined by
\begin{align}
    \label{rho_d}
    d \rho (u) =  \mathcal Z_{K}^{-1} \ind_{ \{ | \int_{\R^d} : | u (x)|^2 : dx | \le K\}}  e^{\frac1 p{\| u \|_{L^p (\R^d)}^p}} d\mu (u)  .
\end{align}

\noi
Furthermore, the resulting measure $\rho$ is absolutely continuous with respect to the base Gaussian free field $\mu$ in \eqref{Gaussian2}.

\smallskip

\noi
\item[\textup{(ii)}] \textup{(critical/supercritical cases)}
Let $ p \ge p^*(d) $ when $d = 2$, and  $p^*(d) \le p < \frac{2d}{d-2} $ when $d \ge 3$.
Then, for any $K >0$, we have
\begin{align}
    \label{non_int_d}
    \sup_{N \in \mathbb N} \mathcal Z_{K,N} = \sup_{N \in \mathbb N} \Big\|  \ind_{ \{ | \int_{\R^d} \wick{| u_N (x)|^2} dx|  \le K\}}  e^{\frac1 p{\| u_N \|_{L^p (\R^d)}^p}} \Big\|_{L^1 (\mu)}   = \infty,
\end{align}

\noi
where $\mathcal Z_{K,N}$ is the partition function given in \eqref{partition2}.
Moreover, 
\begin{align}
    \label{non_int_d2}
    \mathcal Z_K := \int \ind_{ \{ | \int_{\R^d} \wick{| u(x)|^2} dx | \le K\}}  e^{\frac1 p{\| u \|_{L^p (\R^d)}^p}} d \mu   = \infty.
\end{align}
In particular, 
the focusing Gibbs measure \eqref{gibbs} defined on radial functions,
even with a Wick-ordered $L^2$-cutoff,
i.e. \eqref{rho_d}, can not be defined as a probability measure when $p\ge p^*(d)$.
\end{itemize}

\end{theorem}

\noi
When $d = 2$, 
Theorem \ref{THM:main2} - (i) provides an alternative proof of the normalizability result for the focusing measure \eqref{rho_d} by Deng \cite{Deng12}, 
whose proof is based on a Fourier analysis approach developed by Bourgain \cite{BO94}.
Our proof of Theorem \ref{THM:main} - (i), however, is based on the variational approach due to Barashkov-Gubinelli \cite{BG}. 
The main novelty of Theorem \ref{THM:main2} lies in (ii), which implies that the normalizability result by Deng \cite{Deng12} is sharp in the dimension $d=2$.
Furthermore, the higher dimensional $d \ge 3$ cases are also new.

\medskip

The proof of Theorem \ref{THM:main2} is almost identical to that of Theorem \ref{THM:main},
whose proofs will be shown in Sections \ref{SEC:nor}
and \ref{SEC:non}.

\begin{remark}\rm 
(i) Theorem \ref{THM:main2} (ii) only provides the non-normalizability in the critical and supercritical cases up to $p<\frac{2d}{d-2}$ when $d\ge 3$. This is due to the fact that, when $p\ge\frac{2d}{d-2}$, the potential $\int_{\R^d}|u|^pdx$ becomes infinite $\mu$-a.s. and needs to be renormalized. We believe that our construction would still provide the non-normalizability of the Gibbs measure in this case, even after having Wick-ordered the potential energy. 

However, we chose to restrict to a non-singular case in this paper, where only the mass cut-off needs to be renormalized, to stay closer to the previous works \cite{BTTz,Deng12} and complete the picture on normalizability vs non-normalizability in the focusing non-singular case. 

The same remark holds for the defocusing case: we explained in the introduction that the Gibbs measure is well-defined in the non-singular defocusing case due to the upper bound on the density $\exp\big(-\frac1p\int_{\R^d}|u|^pdx\big)$ when $p<\frac{2d}{d-2}$, $d\ge 3$. However, the constructibility of the defocusing Gibbs measure in the singular case $p\ge \frac{2d}{d-2}$ is an interesting open problem.

\smallskip

(ii) Without radial assumption, a renormalization would be needed for any $p>2$ as soon as $d\ge 2$. It would also be very interesting to study the construction of the renormalized Gibbs measure in the defocusing case $d=2$, $p>4$ in order to complete the picture from the results in \cite{dBDF2}. On the other hand, in the focusing case, even with a renormalization of the potential energy, we expect the Gibbs measure to be non-normalizable, similarly as in \cite{BS}. Indeed, we already explained that the Gibbs measure $\rho$ in \eqref{gibbs2} is more singular, for the same dimension, than its finite volume anharmonic counterpart in \eqref{gibbs3}.
\end{remark}

\section{Preliminary}
\label{SEC:subcritical}

\noi
In this section, 
we collect some useful tools.

\subsection{Variational formulation}
In order to prove Theorem \ref{THM:main} and Theorem \ref{THM:main2},
we recall a variational formula for the partition functional 
$\mathcal Z_{p,K}$, similarly to \cite{BG,BG2,OOT1,OOT2,OST,LW}.
Let $W(t)$ denote a cylindrical Brownian motion in $L^2(\R)$ (respectively $L^2_{\textup{rad}}(\R^d)$) defined by
\begin{align}
\label{Bro}
W(t) = \sum_{n\ge 0} B_n (t) h_n,
\end{align}

\noi
where $\{h_n\}_{n\ge 0}$ is the sequence of normalized eigenfunctions of the operator $\mathcal L$ given in \eqref{HP} (or \eqref{HPd} under the radial symmetry assumption) in the previous section, 
and $\{B_n\}_{n \ge 0}$ is a sequence of mutually independent complex-valued Brownian motions.
We define a centered Gaussian process $Y(t)$ by
\begin{align}
\label{Yt}
Y (t) = \L^{-\frac12} W(t) = \sum_{n\ge 0} \frac{B_n (t)}{\ld_n} h_n.
\end{align}

\noi
Then, 
$Y (t)$ is well-defined in $\H^{-s} (\R^d)$ for any $s >0$ and $d \ge 1$, see Corollary \ref{COR:intp} below. 
But $Y(t)$ is not in $L^2(\R^d)$, since we have
\[
\E \big[\|Y(t)\|_{L^2(\R^d)}^2\big] \sim \sum_{n\ge 0} \frac{\E [|B_n (t)|^2] }{n+1}  \|h_n(x)\|_{L^2(\R^d)}^2 = \sum_{n\ge 0} \frac{2 t }{n+1} = \infty
\]

\noi
unless $t = 0$.
From \eqref{Yt}, we see that
\begin{align}
\label{law}
\textup{Law} (Y(1)) = \mu,
\end{align} 
where $\mu$ is the Gaussian free field given in \eqref{mu}.
In what follows, we set $Y_N (1)= \P_{\le N} Y (1)$ and thus $\textup{Law} (Y_N(1)) = (\P_{\le N})_*\mu$,
i.e. the law of the random variable $Y_N(1)$ is the pushforward of $\mu$ under the projection $\P_{\le N}$.
Furthermore, we note that
\[
\E [ |Y_N (1)|^2 ]  = \s_N (x),
\]

\noi
where $\s_N (x) $ is as in \eqref{variance}.

Let $\mathbb{H}_a$ be the space of drifts,
which consists of progressively measurable processes belonging to 
$L^2 ([0,1]; L^2 (\R^d))$,
$\mathbb P$-almost surely.
One of the key tools in this paper is the following Bou\'e-Dupuis variational formula \cite{BD,Ust},
see also \cite[Lemma 3.1]{OST}.

\begin{lemma}
\label{LEM:var}
Let $Y(t)$ be as in \eqref{Yt} and $Y_N (t) = \P_{\le N} Y (t)$.
Suppose that $F: C^\infty (\R^d) \to \R$ is measurable
such that $\E \big[ | F(Y_N (1))|^p\big] < \infty$ and 
$\E \big[ | e^{- F(Y_N (1))}|^q \big] < \infty$ for some $1 < p,q < \infty$ 
with $\frac1p + \frac1q =1$.
Then, we have
\begin{align}
\label{var}
-\log \E \Big[ e^{F(Y_N (1))} \Big] = \inf_{\theta \in \mathbb{H}_a} \E
\bigg[F\big(Y_N(1) + \P_{\le N}I(\theta) (1) \big) + \frac12 \int_0^1 \| \theta (t) \|_{L^2_x (\R^d)}^2 dt\bigg],
\end{align}

\noi
where $I (\theta)$ is defined by
\[
I (\theta) (t) = \int_0^t \L^{-\frac12} \theta (\tau) d\tau
\]

\noi
and
the expectation $\E = \E_{\mathbb P}$ is respect to the underlying probability measure $\mathbb P$.
\end{lemma}

In the following, we shall choose
\[
F(Y_N) =\frac1 p{\| Y_N \|_{L^p (\R^d)}^p}  \ind_{ \{ | \int_{\R^d} : | Y_N (x)|^2 : dx | \le K\}}  ,
\]

\noi
and then the main tasks in Section \ref{SEC:nor} and Section \ref{SEC:non} are to estimate the right hand side of \eqref{var}.

\medskip

We see that $I(\theta)(1)$ enjoys 
the following pathwise regularity bound.

\begin{lemma}
\label{LEM:bounds}
For any $\theta \in \mathbb{H}_a$,
we have
\begin{align}
\label{I}
\| I (\theta) (1)\|^2_{\mathcal H^1 (\R^d)} \le \int_0^1 \| \theta (t) \|_{L^2_x}^2 dt.
\end{align}
\end{lemma}

\begin{proof}
From Minkowski's inequality and  Definition \ref{DEF:sob}, we have
\begin{align*}
    \begin{split}
    \| I (\theta) (1)\|^2_{\mathcal H^1 (\R^d)} & \le \int_0^1 \big\| \mathcal L^{-\frac12} \theta (\tau) \big\|_{\mathcal H^1 (\R^d)} d\tau\\
    & = \int_0^1  \| \theta (\tau) \|_{L^2 (\R^d)} d\tau,
    \end{split}
\end{align*}

\noi
which gives \eqref{I}.
\end{proof}

\subsection{Wiener chaos estimate}
In this subsection, 
we recall the Wiener chaos estimates,
which plays a key role in our analysis in Section \ref{SEC:nor} and Section \ref{SEC:non}.

For convenience, we first recall the Khintchine inequality, which is a special case of the Wiener chaos estimates and has an elementary proof,  see \cite[Lemma 4.2]{BT08}.

\begin{lemma}[Khintchine inequality]
\label{LEM:Khi}
Let $\{g_n (\o)\}$ be a sequence of independent, complex, normalized Gaussian random variables.
Then, there exists $C>0$ such that for all $p\ge 2$ and $\{c_n\} \in \l^2(\mathbb N)$,
\[
\bigg\| \sum_{n \ge 0} c_n g_n (\o) \bigg\|_{L^p (\O)} \le C \sqrt p  \bigg( \sum_{n\ge 0} |c_n|^2 \bigg)^{\frac12}.
\]
\end{lemma}

We then have
the following consequence of Khintchine inequality.
\begin{corollary}
\label{COR:intp}
\textup{(i)}. Given $ 2 < p < \infty$ when $d = 1,2$ and $2 < p < \frac{2d}{d-2}$ when $d \ge 3$, we have
\[
\E \big[\| Y_N (1) \|_{L^p(\R^d)}^p \big] \les_p  1 ,
\]

\noi
where the constant only depends on $p$.

\textup{(ii)}. Let $\dl > 0$ and $1 \le p < \infty$.
Then, we have 
\[
\E \big[\| Y_N (1) \|_{\mathcal H^{-\delta}(\R^d)}^p \big] \les_p  1 ,
\]

\textup{(iii)}. Let $p,d$ be as in \textup{(i)} when $\|\cdot\|_X = \| \cdot \|_{L^p}$, or $p,d,\dl$ as in \textup{(ii)} when $\|\cdot\|_X = \| \cdot \|_{\mathcal H^{-\dl}}$ respectively. Then $Y(1)\in X$, the sequence $Y_N(1)$ is Cauchy in $L^p(\Omega; X)$, and we have that 
\[
\E \big[\| Y(1) - Y_N(1) \|_{X}^p \big] \les_p  N^{-\gamma},
\]
for some $\gamma = \gamma(p,d,\dl) > 0$.
\end{corollary}

\begin{proof}
We start with proving (i).
Recall that $\{B_n(1)\}$ is a sequence of independent, complex, normalized Gaussian random variables.
Therefore, by \eqref{Yt} and Lemma \ref{LEM:Khi} we have
\begin{align*} 
\E [\| Y_N (1) \|_{L^p(\R^d)}^p]  & =  \int_{\R^d} \E \bigg[ \Big| \sum_{n=0}^N \frac{B_n(1) h_n(x)}{\ld_n} \Big|^p \bigg] dx\\
& \les  \int_{\R^d} \bigg(  \sum_{n=0}^N \frac{h_n^2 (x)}{\ld_{n}^2}    \bigg)^{\frac{p}2} dx.\\
\intertext{Since $p > 2$, we use Minkowski inequality to deduce that}
& \les  \bigg(  \sum_{n=0}^N \frac{ \|h_n (x) \|_{L^p (\R^d)}^2 }{\ld_{n}^2}    \bigg)^{\frac{p}2}.  \\
\intertext{We then insert Lemma \ref{LEM:eig1} and Lemma \ref{LEM:eig2} to get}
& \les \bigg(  \sum_{n=0}^\infty \frac{1}{\jb{n}^{1 + \beta(d,p)}}    \bigg)^{\frac{p}2},
\end{align*}

\noi
where 
\begin{align*}
\beta(d,p)=\begin{cases}
\frac13 -\frac2{3p},&~~\text{ if }d=1\text{ and }2\le p\le 4, \\
\frac16,&~~\text{ if }d=1\text{ and }p\ge 4,\\
d(\frac12-\frac1p),&~~\text{ if }d\ge 2\text{ and }2\le p<\frac{2d}{d-1},\\
\frac12-,&~~\text{ if }d\ge 2\text{ and }p=\frac{2d}{d-1},\\
1-d(\frac12-\frac1p),&~~\text{ if }d\ge 2\text{ and }p>\frac{2d}{d-1}.
\end{cases}
\end{align*}
In particular note that $\beta(d,p)>0$ if and only if $p > 2$ when $d=1,2$; or $2 < p < \frac{2d}{d-2}$ when $d\ge 3$.
Thus we finish the proof of (i).

Then, we turn to the proof of (ii).
We only consider the case when $p \ge 2$,
as the case of $p < 2$ follows from H\"older inequality together with the case of $p=2$. 
Similar computation as in (i) with Lemma \ref{LEM:Khi} yields
\begin{align*} 
    \E [\| Y_N (1) \|_{\mathcal H^{-\delta}(\R^d)}^p] & =  \E \bigg[ \Big\| \sum_{n=0}^N \frac{B_n(1) h_n(x)}{\ld_n^{1+\dl}} \Big\|_{L^2(\R^d)}^p \bigg] \\
    & \les  \bigg\| \Big\| \sum_{n=0}^N \frac{B_n(1) h_n(x)}{\ld_n^{1+\dl}} \Big\|_{L^p(\O)} \bigg\|_{L^2(\R^d)}^p \\
    & \les  \bigg\|  \Big( \sum_{n=0}^N \frac{h_n^2(x)}{\ld_n^{2+2\dl}}   \Big)^{\frac12} \bigg\|_{L^2(\R^d)}^p \\
    & = \Big( \sum_{n=0}^\infty \frac{1}{\ld_n^{2+2\dl}}   \Big)^{\frac{p}2}  < \infty,
\end{align*}

\noi
where in the last step we used the fact that $\ld_n^2 \sim 1+ n$ and $\dl>0$. Thus we finish the proof of (ii).

Now we move to the proof of (iii). Let $M > N$, and note that 
$$ Y_M(1) - Y_N(1) = \sum_{n=N+1}^M \frac{B_n(1) h_n(x)}{\lambda_n}.$$
Therefore, proceeding as in the proof of (i)-(ii) respectively, we obtain that 
\begin{equation*}
\E[\|Y_M(1) - Y_N(1)\|_X^p] \les_p 
\begin{cases}
\displaystyle \bigg(  \sum_{n=N+1}^\infty \frac{1}{\jb{n}^{1 + \beta(d,p)}}    \bigg)^{\frac{p}2} \les N^{-\frac{p\beta(d,p)}2} & \text{ in case }X=L^p,\\
\displaystyle \Big( \sum_{n=N+1}^\infty \frac{1}{\ld_n^{2+2\dl}}   \Big)^{\frac{p}2} \les N^{-p\dl} & \text{ in case }X=\H^{-\dl}.
\end{cases}
\end{equation*}
This shows that $Y_N$ is Cauchy in $L^p(\Omega; X)$. We complete the proof by taking the limit in the estimate as $M \to \infty$.
\end{proof}

To state the general version of the Wiener chaos estimate,
we need to recall some basic definitions from stochastic analysis; see \cite{Boga} for instance.
Let $(H,B,\mu)$ be an abstract Wiener space,
where $\mu$ is a Gaussian measure on a separable Banach space $B$ with $H \subset B$ as its Cameron-Martin space.
Let $\{e_j\}_{j\in \mathbb N}$ be a complete orthonormal system of $H^* = H$ such that $\{e_j\}_{j\in \mathbb N} \subset B^*$.
We define a polynomial chaos of order $k$ of the form
\[
\prod_{j=0}^\infty H_{k_j} (\jb{x,e_j}), 
\]

\noi
where $x \in B$, $k_j \neq 0$ for only finitely many $j$'s with $k = \sum_{j=0}^\infty k_j$,
$H_{k_j}$ is the Hermite polynomial of degree $k_j$,
and $\jb{\cdot,\cdot}$ denotes the $B$-$B^*$ duality pairing.
We then define $\mathcal H_k$ as the closure of polynomial chaoses of order $k$ under $L^2(B,\mu)$.
The elements in $\mathcal H_k$ are called homogeneous Wiener chaoses of order $k$. We also write
\[
\mathcal H_{\le k} = \bigoplus_{j=0}^k \mathcal H_j
\]

\noi
for $k\in \mathbb N$.
Then we have the following Wiener chaos estimate,
\begin{lemma}
[Wiener chaos estimate]
\label{LEM:WCE}
Let $k\in \mathbb N$.
Then, we have
\[
\|X\|_{L^p(\O)} \le (p-1)^{\frac{k}2} \|X\|_{L^2(\O)}
\]

\noi
for any $p \ge 2$ and any $X \in \mathcal H_{\le k}$.
\end{lemma}

It is easy to see that the Khintchine inequality Lemma \ref{LEM:Khi} follows from the Wiener chaos estimate Lemma \ref{LEM:WCE} with $k=1$.
The proof of Lemma \ref{LEM:WCE} follows from the hypercontractivity of the Ornstein-Uhlenbeck semigroup. See \cite[Theorem I.22]{Simon} for more details.

Recall the definition of the Wick power $\wick{|Y_N (1)|^2}$ defined in \eqref{Wick},
which are homogeneous Wiener chaos of order $2$.
Therefore, we have the following corollary of Lemma \ref{LEM:WCE},
\begin{corollary}
\label{COR:WCE}
Let $p \ge 1$, we have
\[
\Big\|\int_{\R^d} \wick{|Y_N (1)|^2} dx \Big\|_{L^p(\O)} \les_p 1 ,
\]

\noi
where the constant only depends on $p$. Moreover, the sequence $\int_{\R^d} \wick{|Y_N (1)|^2}dx$ is Cauchy in $L^p(\O)$, and  
$$ 
\Big\|\int_{\R^d} \wick{|Y (1)|^2} dx - \int_{\R^d} \wick{|Y_N (1)|^2} dx \Big\|_{L^p(\O)} \les_p N^{-\frac12}.
$$
\end{corollary}

\begin{proof}
By using H\"older inequality for $p < 2$ and Wiener chaos estimate Lemma \ref{LEM:WCE} for $p \ge 2$, we have
\[
\Big\|\int_{\R^d} \wick{|Y_N (1)|^2} dx \Big\|_{L^p(\O)}  \les_p  \Big\|\int_{\R^d} \wick{|Y_N (1)|^2} dx \Big\|_{L^2(\O)} .
\]

\noi
Then from the definition of Wick order \eqref{Wick} and \eqref{variance},
it follows that
\begin{align*}
\int_{\R^d} \wick{|Y_N (1)|^2} dx & = \int_{\R^d} |Y_N (1)|^2 dx - \int_{\R^d} \s_N (x) dx\\
& = \sum_{n=0}^N \frac{|B_n(1)|^2 - 1}{\ld_n^2}.
\end{align*}

\noi
As the random variable $\{B_n(1)\}$ are normalized and independent, thus we have
\[
\E [(|B_{n}(1)|^2 - 1)^2] = \E [|B_{n}(1)|^4] - 2 [|B_{n}(1)|^2] + 1 = 2 
\]

\noi
and
\[
\E [(|B_{n_1}(1)|^2 - 1)(|B_{n_2}(1)|^2 - 1)] = 0
\]

\noi
for all $n_1 \neq n_2$.
Therefore, we have
\begin{align*}
    \E \bigg[ \Big|\int_{\R^d} & \wick{|Y_N (1)|^2} dx \Big|^2 \bigg] = \sum_{n=0}^N \frac{\E [(|B_n(1)|^2 - 1)^2]}{\ld_n^4}\\
    & = \sum_{n=0}^N \frac{2}{\ld_n^4} <  \sum_{n=0}^\infty \frac{2}{\ld_n^4} < \infty.
\end{align*}

We now move to the second part. Proceeding similarly, for $M > N$, we have that 
\begin{align*}
\Big\|\int_{\R^d} \wick{|Y_M (1)|^2} dx - \int_{\R^d} \wick{|Y_N (1)|^2} dx \Big\|_{L^p(\O)}^p & \les_p 
\Big(\sum_{n=N+1}^M \ld_n^{-4}\Big)^\frac p2 \les N^{-\frac p2},
\end{align*}
which shows that the sequence $\int_{\R^d} \wick{|Y_N (1)|^2}dx$ is Cauchy in $L^p(\O)$. We conclude the proof by taking a limit as $M \to \infty$.
\end{proof}
Finally, in the proof of Theorem \ref{THM:main}-(i) and Theorem \ref{THM:main2}-(ii), in order to show convergence in total variation of the indicator functions, we will need the following Lemma.
\begin{lemma}
Let $\wick{|u_N|^2}$ be the Wick power defined in \eqref{Wick}, and let $\wick{|u|^2}$ be its limit. Then, for every $K \in \R$, and for every $\eps > 0$,
\begin{equation}\label{boundary_estimateN}
\PP\bigg(\int_{\R^d} \wick{|u_N|^2}(x)\,dx \in [K - \eps, K+ \eps]\bigg) \les \eps,
\end{equation}
where the implicit constant is independent of $K \in \R$ and $\eps$. As a consequence, we obtain 
\begin{equation}\label{boundary_estimate}
\PP\bigg(\int_{\R^d}  \wick{|u|^2}(x)\,dx \in [K - \eps, K+ \eps]\bigg) \les \eps.
\end{equation}
\end{lemma}
\begin{proof}
By the $L^2$-orthogonality of the eigenfunctions $h_n$ and the definition \eqref{Wick} of $\wick{|u_N|^2}$, we have that 
$$ \int_{\R^d} \wick{|u_N|^2}(x) dx = \sum_{n=0}^N \frac{|g_n|^2}{\ld_n^2} -  \frac{1}{\ld_n^2} = \frac{|g_0|^2- 1}{d} + \bigg(\sum_{n=1}^N \frac{|g_n|^2-1}{\ld_n^2}\bigg) =: A_0 + A_N. $$
Because of independence of the Gaussians $\{g_n\}_{n\in\N}$, the random variables $A_0$ and $A_N$ are independent as well. Denoting by $d\lambda$ the Lebesgue measure on $\R$, let $\sigma_0 d\lambda$ be the law of $A_0$ and $\sigma_Nd\lambda$ be the law of $A_N$. Since 
$$ \sigma_0 d\lambda = \Law\Big(\frac{|g_0|^2- 1}{d}\Big),$$
and $g_0$ is a normal Gaussian, then $\sigma_0 \in L^\infty(\R)$. Therefore, 
\begin{align*}
\bigg\|\frac{d\Law\Big(  \int_{\R^d} \wick{|u_N|^2}(x) dx\Big)}{d\lambda}\bigg\|_{L^\infty(\R)} &= \Big\|\frac{d\Law( A_0 + A_N)}{d\lambda}\Big\|_{L^\infty(\R)} \\
&= \| \sigma_0 \ast \sigma_N\|_{L^\infty(\R)} \\
&\le \| \sigma_0 \|_{L^\infty(\R)} \|\sigma_N\|_{L^1(\R)} \\
&= \| \sigma_0 \|_{L^\infty(\R)} \\
&\les 1,
\end{align*}
from which \eqref{boundary_estimateN} and hence \eqref{boundary_estimate} follow.
\end{proof}

\section{Normalizability}
\label{SEC:nor}

In this section, 
we show 
the integrability part, i.e. Theorem \ref{THM:main} (i) and Theorem \ref{THM:main2} (i).
Namely,
we prove the boundedness of 
$\mathcal Z_{K,N}$ in \eqref{partition} with $2 < p < 2+ \frac{4}d$ for all $K>0$.
We first start by showing \eqref{uniint_p} and \eqref{uniint_d}. We will then complete the proof of Theorem \ref{THM:main} (i) and Theorem \ref{THM:main2} (i) in the end of this section. The two statements \eqref{uniint_p} and \eqref{uniint_d} correspond to the following bound:
\begin{align}
\label{var1d}
\sup_N \E_{\mu}  \Big[ \exp (R_p(u_N)) \cdot  \ind_{\{ | \int_{\R^d} :|u_N (x)|^2: dx| \le K\}}\Big] < \infty, 
\end{align}

\noi
where $u_N = \P_{\le N} u$ and $R_p(u_N)$ is the potential energy denoted by
\begin{align}
\label{Rp}
R_p (u_N) : = {\frac{r}{p} \int_{\R^d} |u_N (x)|^p\,\mathd  x}.
\end{align}

\noi
Observing that
\begin{align}
\label{var2}
\E_{\mu}  \Big[ \exp (R_p(u_N)) \cdot  \ind_{\{ | \int_{\R^d} :|u_N (x)|^2: dx| \le K\}} \Big]
\le \E_{\mu}  \bigg[ \exp \Big(R_p(u_N) \cdot  \ind_{\{ | \int_{\R^d} :|u_N (x)|^2: dx| \le K\}} \Big)\bigg],
\end{align}

\noi
then the bound \eqref{var1d} follows once we have
\begin{align}
\label{var3}
\begin{split}
\sup_N   \E_{\mu}  \bigg[ \exp \Big(R_p(u_N) \cdot  \ind_{\{ | \int_{\R^d} :|u_N (x)|^2: dx| \le K\}} \Big)\bigg]  < \infty.
 \end{split}
\end{align}

\noi
From \eqref{law} and the Bou\'e-Dupuis variational formula Lemma \ref{LEM:var},
it follows that
\begin{align}
\label{var4}
\begin{split}
& -  \log  \E_{\mu}  \Big[ \exp \Big(R_p(u_N) \cdot  \ind_{\{ | \int_{\R^d} :|u_N |^2: dx| \le K\}}  \Big)\Big]\\
& \hphantom{X}  = -  \log  \E  \Big[ \exp \Big( R_p(Y_N  ) \cdot  \ind_{\{ | \int_{\R^d} :|Y_N |^2: dx| \le K\}}  \Big)\Big] \\
&\hphantom{X} = \inf_{\theta \in \mathbb{H}_a} \E \Big[ - R_p \big( Y_N + \P_{\le N} I(\theta) (1) \big) 
\cdot  \ind_{\{ | \int_{\R^d} :| Y_N  + \P_{\le N} I(\theta) (1) |^2: dx|  \le K \}} \\
& \hphantom{XXXXXXX} + \frac12 \int_0^1 
 \| \theta (t) \|_{L^2_x}^2 dt  \Big], 
\end{split}
\end{align}

\noi
where $Y_N = \P_{\le N} Y (1)$ and $Y(1)$ is given in \eqref{Yt}.
Here, $\E_{\mu}$ and $\E$ denote expectations with respect to the Gaussian field $\mu$ and the underlying probability measure $\mathbb P$ respectively.
In what follows, we will denote
\begin{align}  
\Theta_N = \P_{\le N} I(\theta)(1)
\label{nots}
\end{align}

\noi 
for simplicity.
In the rest of this section, 
we show that the right hand side of \eqref{var4} 
has a finite lower bound.

\medskip

\begin{proof}[Proof of \eqref{var1d}]
In this case, we prove \eqref{var3} 
with a Wick-ordered $L^2$-cutoff of any finite size $K$.
By duality and Young's inequality, 
we have
\begin{align}
\begin{split}
\Big| \int_{\R^d} & \wick{| Y_N + \Theta_N|^2} dx \Big|  \\
& = \Big| \int_{\R^d} \wick{|Y_N|^2}  dx  + 2 \int_{\R^d}  Y_N \cj{\Theta_N}  dx + \int_{\R^d} |\Theta_N|^2 dx \Big| \\
& \ge -\, \Big| \int_{\R^d} \wick{|Y_N|^2}  dx \Big|  - 2 \|Y_N \|_{\H^{-\dl}} \| \Theta_N \|_{\H^\dl} + \int_{\R^d}  |\Theta_N|^2 dx\\
& \ge  -\Big| \int_{\R^d} \wick{|Y_N|^2}  dx \Big|  - C_\eps \| Y_N \|_{\H^{-\dl} ({\R^d})}^{p_1} -  \eps \| \Theta_N \|_{\H^\dl ({\R^d})}^{q_1} + \int_{\R^d} |\Theta_N|^2  dx,
\end{split}
\label{cutoff0}
\end{align}

\noi
where $\dl > 0$ and $p_1, q_1 > 1$ such that
\[
\frac1{p_1} + \frac1{q_1} = 1.
\]

\noi
Furthermore, by interpolation, we have
\[
 \| \Theta_N \|_{\H^\dl ({\R^d})}
\le  C  \| \Theta_N \|_{L^2 ({\R^d})}^{(1-\dl)p_2} +  C \| \Theta_N \|_{\H^1 ({\R^d})}^{\dl q_2},
\]

\noi
where we choose $\dl \in (0,1)$ and $p_2, q_2 > 1$ such that
\[
\frac1{p_2} + \frac1{q_2} = 1.
\]

\noi
We may then choose $ \dl \ll1$, $q_1 = 1+$,
and $p_2 = 2+$, such that $q_1(1-\dl)p_2 = 2$.
It also follows that 
\begin{align}
\dl q_1 q_2 \ll 1
\label{smalldelta} 
\end{align}

\noi
as $q_2 = 2-$. 
Therefore, with the parameters chosen as above we have
\begin{align}
\label{Lqbound}
\| \Theta_N \|_{\H^\dl ({\R^d})}^{q_1} 
\le   C \| \Theta_N \|_{L^2 ({\R^d})}^{2} +  C \| \Theta_N \|_{\H^1 ({\R^d})}^{\dl q_1 q_2}.
\end{align}

\noi
We remark here that the constant $C$ is independent of $N$ and $\eps$, 
and may vary from line to line.
By choosing $\eps C < \frac12$, 
from \eqref{cutoff0} and \eqref{Lqbound}, 
we conclude that
%
\begin{align}
\begin{split}
\bigg\{  \Big| \int_{\R^d} & :| Y_N + \Theta_N|^2: dx \Big|   \le K \bigg\}  \\
& = \bigg\{ \Big| \int_{\R^d} : |Y_N|^2 \!:  dx  + 2 \int_{\R^d}  Y_N \cj{\Theta_N}  dx + \int_{\R^d} |\Theta_N|^2 dx \Big|  \le K \bigg\} \\
& \subset \bigg\{  \|\Theta_N \|_{L^2(\R^d)}^2 \le   K + \Big| \int_{\R^d} : |Y_N|^2 \!:  dx   \Big| + C_\eps \| Y_N \|_{\H^{-\dl} (\R^d)}^{p_1} \\
& \hphantom{XXXX}  + \frac12 \| \Theta_N \|_{L^2 (\R^d)}^{2} + \eps C  \| \Theta_N \|_{\H^1 (\R^d)}^{\dl q_1 q_2} \bigg\}\\
& = \bigg\{  \|\Theta_N \|_{L^2(\R^d)}^2 \le  2 K + 2 \Big| \int_{\R^d} : |Y_N|^2 \!:  dx   \Big| + 2 C_\eps \| Y_N \|_{\H^{-\dl} (\R^d)}^{p_1} \\
& \hphantom{XXXX}   + 2 \eps C \| \Theta_N \|_{\H^1 (\R^d)}^{\dl q_1 q_2} \bigg\}\\
& = : \O_K.
\end{split}
\label{cutoff2}
\end{align}

\noi
We then recall an elementary inequality, 
which is a direct consequence of the mean value theorem and 
 the Young's inequality.
Given $p >2$ and $\eps >0$, there exists $C_\eps$ such that
\begin{align}
\label{Young}
|z_1 + z_2|^p \le (1+ \eps)|z_1|^p + C_\eps |z_2|^p  
\end{align}

\noi
holds uniformly in $z_1,z_2 \in \mathbb C$.
Here $C_\eps$, 
which may differ from line to line,
denotes a constant depending on $\eps$ but not $N$.
We conclude from \eqref{Rp}, \eqref{Young}, \eqref{cutoff2},
\eqref{GNSdisp1}, and \eqref{GNSdisp2}, that
\begin{align}
& R_p \big(Y_N + \Theta_N\big) 
\cdot  \ind_{\{ | \int_{\R^d} : |Y_N + \Theta_N|^2 : dx | \le K  \}} \notag\\
&  \le (1+ \eps) R_p \big(\Theta_N\big) 
\cdot  \ind_{\{ | \int_{\R^d} : |Y_N + \Theta_N|^2 : dx | \le K  \}} + C_\eps R_p(Y_N)  \notag\\
& \le (1+ \eps) R_p \big( \Theta_N \big) 
\cdot  \ind_{\O_K}  + C_\eps R_p(Y_N) \notag\\
& \le \frac{1+\eps}p C_{\textup{GNS}} \| \Theta_N\|^{\frac{(p-2)d}{2}}_{\H^1 (\R^d)}\|\Theta_N\|^{2+\frac{2-d}2 (p-2)}_{L^2(\R^d)} \cdot \ind_{\O_K}   + C_\eps R_p(Y_N) \notag\\
\intertext{where $C_{\textup{GNS}}$ is the implicit constant in \eqref{GNSdisp1} or \eqref{GNSdisp2} (depending on the dimension) and the set $\O_K$ is given in \eqref{cutoff2}. 
Noting that $\frac{(p-2)d}{2} < 2$ when $p<  p^*(d) = 2+ \frac{4}d $, 
we apply Young's inequality to continue with}
& \le  C \| \Theta_N  \|_{L^2({\R^d})}^{2+\frac{4(p-2)}{4+2d-dp}} \cdot \ind_{\O_K}
+ \frac14 \| \Theta_N  \|_{\mathcal H^1 ({\R^d})}^2 
 + C_\eps R_p(Y_N).
\label{var5}
\end{align}

\noi
Then from \eqref{cutoff2}, interpolation, and the Young's inequality, 
we have
\begin{align}
\begin{split}
 &  \| \Theta_N \|_{L^2({\R^d})}^{2+\frac{4(p-2)}{4+2d-dp}} \cdot \ind_{\O_K}  \\
& \les C  K^{1+\frac{2(p-2)}{4+2d-dp}} + C  \Big| \int_{{\R^d}} \wick{|Y_N|^2}  dx   \Big|^{1+\frac{2(p-2)}{4+2d-dp}}  +   
CC_\eps  \| Y_N \|_{\H^{-\dl} ({\R^d})}^{p_1 \big(1+\frac{2(p-2)}{4+2d-dp} \big)} \\
& \hphantom{XXXX}  + \big( \eps   C  \| \Theta_N \|_{\H^1 ({\R^d})}^{\dl q_1 q_2}  \big)^{1+\frac{2(p-2)}{4+2d-dp}}    \\
& \le C_K +  C   \Big| \int_{{\R^d}} \wick{|Y_N|^2}  dx   \Big|^{1+\frac{2(p-2)}{4+2d-dp}}   + CC_\eps   \| Y_N \|_{\H^{-\dl} ({\R^d})}^{p_1 \big(1+\frac{2(p-2)}{4+2d-dp} \big)} 
 +   \eps C  \| \Theta_N \|_{\H^1 ({\R^d})}^{2}  ,
\end{split}
\label{cutoff3}
\end{align}

\noi
where $\dl, \eps \ll 1$ and $q_1, q_2 \in \R$ are real numbers such that \eqref{smalldelta} holds.
Here we remark that the constants $C,C_K$ in \eqref{cutoff3}, which may differ from line to line, are chosen to be independent of $\eps$ (but the latter depends on $K$).
From Corollary \ref{COR:intp}, 
we have
\begin{align}
\label{bound1}
\begin{split}
 \sup_{N}  \E \Big[ \Big| \int_{{\R^d}} \wick{|Y_N|^2}  dx   \Big|^{1+\frac{2(p-2)}{4+2d-dp}} \Big] + \sup_N  \E \Big[  \| Y_N \|_{\H^{-\dl} ({\R^d})}^{p_1 \big(1+\frac{2(p-2)}{4+2d-dp} \big)} \Big]  < \infty,
 \end{split}
\end{align}

\noi
provided $2 < p < p^*(d) = 2 + \frac4d$.
By collecting \eqref{var4}, \eqref{var5}, \eqref{cutoff3}, and \eqref{bound1}, and Lemma \ref{LEM:bounds},
we arrive at
 \begin{align}
\label{var70}
\begin{split}
- & \log  \E_{\mu}  \bigg[ \exp \Big(R_p(u_N) \cdot  \ind_{\{ | \int_{\R^d} :|u_N (x)|^2: dx| \le K\}} \Big) \bigg] \\
& \ge \inf_{\theta \in \mathbb{H}_a} \E \bigg[ - C \| \Theta_N \|_{L^2({\R^d})}^{2+\frac{4(p-2)}{4+2d-dp}} \cdot \ind_{\O_K}
  - C  R_p(Y_N) \\
& \hphantom{XXXXXXXX} 
- \frac14 \| \Theta_N \|_{\mathcal H^1 ({\R^d})}^2 + \frac12 \int_0^1 
 \| \theta (t) \|_{L^2_x}^2 dt \bigg]\\
& \ge \inf_{\theta \in \mathbb{H}_a}  \E \bigg[  - C    \| \Theta_N \|_{L^2({\R^d})}^{2+\frac{4(p-2)}{4+2d-dp}} \cdot \ind_{\O_K} + \frac14  \int_0^1 
  \| \theta (t) \|_{L^2_x}^2 dt 
  - C R_p(Y_N) \bigg]\\
& \ge \inf_{\theta \in \mathbb{H}_a}  \E \bigg[ - C_K - C   \Big| \int_{{\R^d}} : |Y_N|^2 \!:  dx   \Big|^{1+\frac{2(p-2)}{4+2d-dp}} - CC_\eps   \| Y_N \|_{\H^{-\dl} ({\R^d})}^{p_1 \big(1+\frac{2(p-2)}{4+2d-dp} \big)} \\
& \hphantom{XXXXXXXX} 
+ \Big( \frac14 - C\eps \Big) \int_0^1 
  \| \theta (t) \|_{L^2_x}^2 dt 
  - C R_p(Y_N)  \bigg].
\end{split}
\end{align}

\noi
Then, by choosing $\eps >0$ sufficiently small such that  $C \eps  < \frac14$, from \eqref{bound1} and \eqref{var70}, we obtain 
\begin{align}
\begin{split}
    - & \log  \E_{\mu}  \bigg[ \exp \Big(R_p(u_N) \cdot  \ind_{\{ | \int_{\R^d} :|u_N (x)|^2: dx| \le K\}} \Big) \bigg] \\
 & \ge  - C_K   - C  \E\big[ R_p(Y_N) \big]  
- C   \E \bigg[ \Big| \int_{{\R^d}} \wick{|Y_N|^2}  dx   \Big|^{1+\frac{2(p-2)}{4+2d-dp}} \bigg]  \\
& \hphantom{XXXXXX} - CC_\eps   \E \bigg[  \| Y_N \|_{\H^{-\dl} ({\R^d})}^{p_1 \big(1+\frac{2(p-2)}{4+2d-dp} \big)} \bigg] \\
& \ge   - C_K   - C \E\big[ R_p(Y_N) \big] .
    \end{split}
    \label{var71}
\end{align}

\noi
We note that $p^*(d) < \infty$ for $d = 1,2$, and $p^*(d) < \frac{2d}{d-2}$ when $d \ge 3$.
Therefore, Corollary \ref{COR:intp} can be applied for $p < p^*(d)$. 
By applying Corollary \ref{COR:intp} with $p < p^*(d)$, we conclude that \eqref{var70} is uniformly bounded from below.
Therefore, we finish the proof of \eqref{var3} and so of \eqref{uniint_p} and \eqref{uniint_d}.
\end{proof}

We now complete the proof of Theorem \ref{THM:main} - (i) and Theorem \ref{THM:main2} - (i).
\begin{proof}[Proof of Theorem \ref{THM:main} - {\rm{(i)}} and Theorem \ref{THM:main2} - {\rm{(i)}}]
First of all, notice that the quantities 
$$ \ind_{ \{ | \int_{\R^d} : | u_N (x)|^2 : dx | \le K\}}  e^{\frac1 p{\| u_N \|_{L^p (\R^d)}^p}} $$
are equiintegrable in $L^r(\mu)$. This is a direct consequence of \eqref{uniint_p} and \eqref{uniint_d} applied to some $\wt r > r$ (and the fact that $\mu$ is a probability measure). 
Therefore, by Vitali's convergence theorem, we just need to show that 
\begin{gather}
\lim_{N \to \infty} e^{\frac1 p{\| u_N \|_{L^p (\R^d)}^p}} = e^{\frac1 p{\| u \|_{L^p (\R^d)}^p}}, \label{lim_meas1}\\
\lim_{N \to \infty} \ind_{ \{ | \int_{\R^d} : | u_N (x)|^2 : dx | \le K\}} =  \ind_{ \{ | \int_{\R^d} : | u (x)|^2 : dx | \le K\}}, \label{lim_meas2}
\end{gather}
where the limits are intended as limits in probability. We have that 
\begin{itemize}
\item \eqref{lim_meas1} is a direct consequence of the convergence of Corollary \ref{COR:intp}, (iii) with $Y_N(1), Y(1)$ replaced by $u_N, u$ respectively.
\item \eqref{lim_meas2} follows from the convergence of $\int_{\R^d} \wick{| u_N (x)|^2} dx$ to $\int_{\R^d} \wick{| u(x)|^2} dx$ and \eqref{boundary_estimate}.
\end{itemize}

\end{proof}

\section{Non-normalizability}
\label{SEC:non}

In this section, we present the proof of  Theorem \ref{THM:main} - (ii) and Theorem \ref{THM:main2} - (ii).
Namely, we prove 
the divergence
\eqref{non_int} with $p\ge p^*(d) = 2 + \frac4d$
for any cut-off size $K>0$. Most of this section is dedicated to showing \eqref{non_int} and \eqref{non_int_d}, we will prove \eqref{non_int2} and \eqref{non_int_d2} at the end of the section.

In order to show \eqref{non_int} and \eqref{non_int_d}, it suffices to prove 
the following divergence:
given $p \ge p^*(d)$, for every $r > 0$,\footnote{Strictly speaking, we only need $r=1$ in order to show \eqref{non_int} and \eqref{non_int_d}. However, the general case $r>0$ will be helpful in showing \eqref{non_int2} and \eqref{non_int_d2}.}
\begin{align}
\lim_{N\to \infty}  \E_{\mu}  \Big[ \exp (R_p(u_N)) \cdot  \ind_{\{ | \int_{\R^d} :|u_N (x)|^2: dx| \le K\}}\Big] =  \infty,
\label{pax}
\end{align}

\noi
where $R_p(u)$ is as  in \eqref{Rp} and $u_N = \P_N u$.

First we notice that 
\begin{align}
\begin{split}
\E_\mu\Big[ & \exp\big(  R_p(u_N) \big) 
\cdot\ind_{\{ |\int_{{\R^d}} \, : |u_N|^2 :\, dx | \le K\}} \Big]\\
&\ge \E_\mu\Big[\exp\Big(  R_p(u_N)
\cdot \ind_{\{ |\int_{{\R^d}} \, : |u_N|^2 :\, dx | \le K\}}\Big)   \Big]
- 1.
\end{split}
\label{pax1}
\end{align}

\noi
Therefore, the  divergence \eqref{pax} follows once we prove
\begin{align}
\lim_{N\to \infty}  \E_\mu\Big[\exp\Big(  R_p(u_N)
\cdot \ind_{\{ |\int_{{\R^d}} \, : |u_N|^2 :\, dx | \le K\}}\Big)   \Big] =  \infty,
\label{pa0}
\end{align}

\noi
for $p \ge p^*(d)$.

By  the  Bou\'e-Dupuis variational formula 
(Lemma \ref{LEM:var}), we have
\begin{align}
\begin{split}
- \log & {\E_\mu\Big[\exp\Big(  R_p(u_N)
\cdot \ind_{\{ |\int_{{\R^d}} \, : |u_N|^2 :\, dx | \le K\}}\Big)   \Big]} \\
&= \inf_{\dr \in \mathbb H_a} \E\bigg[ -  R_p (Y_N + \Theta_N)  \ind_{\{ |\int_{{\R^d}} : | Y_N|^2: 
+ 2\Re (Y_N  \cj{\Theta_N} ) + |\Theta_N|^2 dx | \le K\}} \\
&\hphantom{XXXXX}
+ \frac 12 \int_0^1 \| \dr(t)\|_{L^2_x} ^2 dt \bigg],
\label{DPf}
\end{split}
\end{align}

\noi
where $Y_N$ and $\Theta_N$ are as in \eqref{nots}.
Here,  $\E_\mu$ and $\E$ denote expectations
with respect to the Gaussian field~$\mu$ in \eqref{mu}
and the underlying probability measure $\PP$, respectively.
In the following, we show that the right-hand side 
of \eqref{DPf} tends to $-\infty$ as $N  \to \infty$ as long as $p \ge p^*(d)$.
As in \cite{TW,OOT1,OOT2,OST}, 
the main idea is to construct a series of drifts $\dr \in \mathbb H_a$
such that 
$\Theta_N$ looks like ``$- Y_N + $ 
a perturbation'', where the perturbation terms are  approximations of a sequence of blow-up profiles,
which are bounded in $L^2$
but have large $L^p$-norm for $p > 2$.
In particular, when $p \ge p^*(d)$, the $L^p$-norm of the sequence of perturbations grow fast enough driving \eqref{DPf} to diverge as $N \to \infty$. 
The key observation is that the Wick-ordered $L^2$-cutoff $|\int_{{\R^d}} \wick{|u_N|^2}\, dx | \le K$ does not exclude such blow-up profiles for any given $K >0$.

\subsection{Blow-up profiles}
In this subsection, 
we first construct a profile which stays bounded in $L^2$ but grows in $L^p$ with $p > 2$.

Fix a large parameter $M \gg 1$.
Let $f: {\R^d} \to \R$ be a real-valued radial Schwartz function
with $\|f\|_{L^2 (\R^d)} = 1$
such that 
the Fourier transform $\ft f$ is a smooth function
supported  on $\big\{\frac 12 <  |\xi| \le 1 \big\}$.
Define a function $f_M$  on ${\R^d}$ by 
\begin{align}
f_M(x) =  M^{-\frac{d}2} \int_{\R^d} e^{ix\cdot \xi} \ft f(\tfrac{\xi}M) d\xi = M^{\frac{d}2} f(Mx), 
\label{fMdef} 
\end{align}

\noi
where $\ft f$ denotes the Fourier transform on ${\R^d}$ defined by 
\[ \ft f(\xi) = c \int_{{\R^d}} f(x) e^{-i\xi \cdot x} dx.\]

\noi
Then, a direct computation yields  the following lemma. 
See Lemma 5.12 in \cite{OOT1} and Lemma 3.3 in \cite{OST} for similar constructions on the torus $\T^d$. 

\begin{lemma}
\label{LEM:soliton}
Let $s \in \R$. Then, we have
\begin{align}
\int_{{\R^d}} f_M^2 dx &= 1 , \label{fM0} \\
\int_{{\R^d}} (\L^{\frac{s}2} f_M)^2 dx &\les M^{2s}, \label{fm2} \\
\int_{{\R^d}}  |f_M|^p  dx &\sim M^{\frac{pd}2-d} \label{fM1}
\end{align}

\noi
for any $p>0$, $M \gg 1$.
\end{lemma}

\begin{proof}
It is easy to see that \eqref{fM0} and \eqref{fM1} follow directly from the definition \eqref{fMdef}.

Now we turn to \eqref{fm2}.
We first consider the case $s \le 0$.
From \eqref{sobolev_1} and \eqref{sobolev_d}, we have 
\[
\| u \|_{H^{-s} (\R^d)} \les \|u\|_{\H^{-s} (\R^d)}.
\]

\noi 
Then the desired estimate \eqref{fm2} for $s \le 0$ follows from duality.
As for $s > 0$,
we have
\[
\begin{split}
\|\L^{\frac{s}2} f_M\|_{L^2({\R^d})} & \les \|\jb{\partial_x}^s f_M\|_{L^2({\R^d})} + \| \jb{x}^s f_M\|_{L^2({\R^d})} \\
& \les M^s \| f_M\|_{L^2({\R^d})} + \|M^{\frac{d}2}  \jb{x}^s f(Mx)\|_{L^2({\R^d})}\\
& \les M^s \|f_M\|_{L^2({\R^d})} +  \|   \jb{M^{-1}x}^s f(x)\|_{L^2({\R^d})} \\
& \les M^s +1,
\end{split}
\]

\noi
provided $M \ge 1$.
Therefore we finish the proof of \eqref{fm2} and thus the lemma. 
\end{proof}

\subsection{Construction of the drift terms}
In the next lemma, we construct
an approximation $Z_M (t)$ to $Y_N (t)$ in \eqref{nots}
by solving stochastic differential equations.
Note that there are similar stochastic approximations in  \cite[Lemma 3.4]{OST1} and \cite[Lemma 3.5]{LW} on $\T^d$ with the standard Laplacian.

\begin{lemma} 
\label{LEM:approx}
Given $ N \ge M\gg 1$,  define $Z_M$ by its coefficients in the eigenfunction expansion of $\L$
as follows.
For $n \leq M$, $\wt Z_M(n, t)$ is a solution of the following  differential equation\textup{:}
\begin{align}
\begin{cases}
d \wt Z_{M}(n, t) = \ld_n^{-1} \sqrt M  (\wt Y_N (n,t)- \wt Z_{M}(n, t)) dt \\
\wt Z_{M}|_{t = 0} =0, 
\end{cases}
\label{ZZZ}
\end{align}
where 
$$ \wt Y_N (n,t) = \int_{\R^d} Y_N(t,x) h_n(x) dx, $$
and we set $\wt Z_{M}(n, t)  \equiv 0$ for $n > M$.
Then, 
$$Z_M(t) = \sum_{n \le M} \wt Z_M(n,t) h_n(x)$$ 
is a centered Gaussian process in $L^2({\R^d})$, which satisfies $\P_{\le M} Z_M = Z_M$. 
Moreover, we have   
\begin{align}
&\E \big[ \|Z_M \|_{L^2({\R^d})}^2 \big]  \sim \log M,\label{NRZ0}\\
&\E\bigg[  2 \Re \int_{{\R^d}} Y_N \cj{Z_M} dx - \int_{{\R^d}} |Z_M|^2 dx   \bigg] \sim \log M, \label{NRZ1}\\
&\E \bigg[  \Big|  \wick{\| Y_N-Z_M\|_{L^2({\R^d})}^2}  \Big|^2      \bigg] \les M^{-1} \log M,    \label{NRZ3}\\
&\E\bigg[\Big| \int_{{\R^d}} Y_N  f_M dx \Big|^2\bigg] 
+ \E\bigg[\Big| \int_{{\R^d}} Z_M  f_M dx \Big|^2\bigg] \les M^{-\frac32},   \label{NRZ5}\\
&\E\bigg[\int_0^1 \Big\| \frac d {ds} Z_M(s) \Big\|^2_{\H^1 (\R^d)}ds\bigg] \les M \label{NRZ6}
\end{align}
	
\noi
for any $N \ge M \gg 1$, 
where  $Z_M =Z_M (1)$
and
\begin{align}
 \wick{ \| Y_N-Z_M\|_{L^2({\R^d})}^2} = 
 \| Y_N-Z_M\|_{L^2({\R^d})}^2 - \E\big[ \| Y_N-Z_M\|_{L^2({\R^d})}^2 \big].
\label{ZZZ2}
\end{align}

\end{lemma}

\begin{proof}[Proof of Lemma \ref{LEM:approx}]
Let 
\begin{align}
X_n(t)=\wt Y_N(n, t)- \wt Z_{M}(n, t), 
\quad 0\le n \le M.
\label{ZZ1} 
\end{align}
Then,  from \eqref{Yt} and \eqref{ZZZ}, 
we see that $X_n(t)$ satisfies 
the following stochastic differential equation:
\begin{align*}
\begin{cases}
dX_n(t)=- \ld_n^{-1} \sqrt M X_n(t) dt + \ld_n^{-1} dB_n(t)\\
X_n(0)=0
\end{cases}
\end{align*}	

\noi
for $0\le n \le M$.
By solving this stochastic differential equation, we have
\begin{align}
X_n(t)= \ld_n^{-1} \int_0^t e^{-  \ld_n^{-1} \sqrt M (t-s)}dB_n(s).
\label{ZZ2}
\end{align}

\noi
Then, from \eqref{ZZ1} and \eqref{ZZ2}, we have 
\begin{align}
\wt Z_{M}(n, t)= \wt Y_N(n, t)- \ld_n^{-1} \int_0^t  e^{-  \ld_n^{-1} \sqrt M (t-s)} dB_n(s)
\label{SDE1}
\end{align}

\noi
for $n \le M$. Hence, from \eqref{SDE1}, the independence of $\{B_n \}_{n \in \N}$,
Ito's isometry, and \eqref{Yt}, we have
\begin{align}
\begin{split}
\E \big[ \|Z_M\|_{L^2 ({\R^d})}^2 \big]&=\sum_{n \le M} \bigg( \E \big[  | \ft Y_N(n) |^2  \big]
- 2 \ld_n^{- 2} \int_0^1 e^{- \ld_n^{-1} \sqrt  M (1-s)}ds \\
& \hphantom{XXXXX}
+ \ld_n^{- 2} \int_0^1 e^{-2 \ld_n^{-1} \sqrt  M (1-s)}ds \bigg)\\
& \sim \log M + O\Big(\sum_{n \le M } \ld_n^{-1} M^{-\frac12}\Big)\\
& \sim \log M
\end{split}
\label{ZZ3}
\end{align}

\noi
for any $M\gg 1$.
This proves \eqref{NRZ0}.

By the $L^2$ orthogonality of $\{h_n\}_{n\in\N}$,  \eqref{SDE1}, \eqref{NRZ0}, and  
proceeding as in \eqref{ZZ3}, we have
\begin{align*}
\E\bigg[  2 \Re \int_{{\R^d}} Y_N & \cj{Z_M} dx - \int_{{\R^d}} |Z_M|^2 dx   \bigg]
=\E \bigg[ 2 \Re \sum_{n \le M }\wt Y_N(n)  \cj{ \wt Z_M(n)} -\sum_{n \le M }|  \wt Z_M(n) |^2    \bigg]\\
&=\E \bigg[ \sum_{n \le M } | \wt Z_M(n)  |^2+ \sum_{n \le M } \Re \bigg( 2 \ld_n^{-1}  \int_0^1 e^{- \ld_n^{-1} \sqrt M (1-s) }dB_n(s) \bigg)  \cj{\wt Z_M(n)}   \bigg]\\
&\sim \log M+O\Big(\sum_{n \le M } \ld_n^{-1}  M^{-\frac12}\Big)\\
&\sim \log M 
\end{align*} 

\noi 
for any $N\ge M\gg 1$.
This proves  \eqref{NRZ1}.

Note that  $\wt Y_N(n)-\wt Z_M(n)$ is a mean-zero Gaussian random variable.
Then, from \eqref{SDE1}
and Ito's isometry, 
we have 
\begin{align}
\begin{split}
\E \bigg[ \Big( & |\wt Y_N(n)-  \wt Z_M(n)|^2-\E\big[ |\wt Y_N(n)-\wt Z_M(n)|^2 \big] \Big)^2  \bigg]\\
& \les 
\Big(\E\big[ |\wt Y_N(n)-\wt Z_M(n)|^2 \big] \Big)^2  \\
& \sim \ld_n^{-4} \bigg(\int_0^1 e^{-2 \ld_n^{-1} \sqrt M(1-s)  } ds \bigg)^2 \\
&\sim \ld^{-2}_n M^{-1},
\end{split}
\label{ZZ4}
\end{align}

\noi
for $0 \le n \le M$.
Hence, 
from Plancherel's theorem, \eqref{ZZZ2}, 
 the independence of $\{B_n \}_{n \in \N}$, 
 the independence
of 
$\big\{ |\wt Y_N(n) |^2- \E \big[ | \wt Y_N(n)|^2 \big]\big\}_{M < n \le N}$
and 
\[\big\{ |\wt Y_N(n)-\wt Z_M(n)|^2-\E\big[ |\wt Y_N(n)-\wt Z_M(n)|^2\big]\big\}_{n \le M},\]

\noi 
 \eqref{Yt}, 
 and \eqref{ZZ4}, we have
\begin{align*}
\E \Big[  & \big|  \wick{\| Y_N-Z_M\|_{L^2({\R^d})}^2}  \big|^2      \Big] \notag \\
&= \sum_{M< n \le N } \E \bigg[ \Big( |\wt Y_N(n) |^2- \E \big[ | \wt Y_N(n)|^2 \big] \Big)^2 \bigg]\notag \\
&\hphantom{XX}+ \sum_{n \le M} \E \bigg[ \Big( |\wt Y_N(n)-\wt Z_M(n)|^2-\E\big[ |\wt Y_N(n)-\wt Z_M(n)|^2 \big] \Big)^2  \bigg]\notag  \\
&\les \sum_{M< n \le N} {\ld_n^{-4}}
+\sum_{ n \le M }
\ld_n^{-2} M^{-1} 
\les M^{- 1} \log M.
\end{align*}

\noi
This proves \eqref{NRZ3}.

From \eqref{fm2} and \eqref{Yt}, we have
\begin{align}
\begin{split}
\E\bigg[ \Big| \int_{{\R^d}} Y_N  f_M dx  \Big|^2\bigg]
&= \E \bigg[ \Big| \sum_{n \le N}  \wt Y_N(n) { \jb{ f_M, h_n}} \Big|^2 \bigg]
= \sum_{n \le N} \ld_n^{-2} |\jb{ f_M, h_n}|^2 \\
&\le \int_{{\R^d}} \big|\L^{-\frac12} f_M (x)\big|^2 dx
\les M^{-2},
\end{split}
\label{app4}
\end{align}

\noi
where in the last step we used Lemma \ref{LEM:soliton}.
From \eqref{ZZ2}, Ito's isometry, and \eqref{fm2}, we have 
\begin{align}
\begin{split}
 \E \bigg[ \Big| \sum_{n \le M} X_n(1) \cj{ \jb{ f_M, h_n} } \Big|^2 \bigg]
&=\E \Bigg[ \bigg|\sum_{n \le M} \bigg(\ld_n^{-1}  \int_0^1 e^{- \ld_n^{-1} \sqrt M(1-s) } dB_n(s) \bigg)  {\jb{ f_M, h_n}}     \bigg|^2     \Bigg]\\
&\les M^{- \frac12} \sum_{n \le M}  \ld_n^{-1} | \jb{ f_M, h_n}|^2\\
&\les M^{- \frac12} \|\L^{-\frac12} f_M\|_{L^2(\R^d)}^2\\
& \les M^{-\frac32}.
\end{split}
\label{app5}
\end{align}

\noi
Hence, \eqref{NRZ5} follows
from \eqref{app4} and \eqref{app5}
with \eqref{SDE1}.

Lastly, from \eqref{ZZZ}, \eqref{ZZ1}, \eqref{ZZ2}, and Ito's isometry, we have 
\begin{align*}
\E\bigg[\int_0^1 \Big\| \frac d {ds} Z_M(s) \Big\|^2_{\H^1}ds\bigg] &= M \E\bigg[\int_0^1 \Big\| \P_{\le M}(Y_N(s)) - Z_M(s) \Big\|^2_{L^2}ds\bigg] \\
&= M \E\bigg[ \int_0^1  \sum_{n \le M} |X_n(s)|^2   ds\bigg]\\
&=  M \sum_{n \le M} \int_0^1\E\big[|X_n(s)|^2\big] ds \\
&=M \sum_{n \le M} \ld_n^{-2} \int_0^1 \int_0^s e^{-2 \ld_n^{-1} \sqrt M (s-s')} d s' ds \\
& \les M \sum_{n \le M } \ld_n^{-1} M^{-\frac12} \\
&\les M, 
\end{align*}

\noi
yielding  \eqref{NRZ6}.
This  completes the proof of Lemma~\ref{LEM:approx}.
\end{proof}

We now define $ \al_{M, N}$ by
\begin{align} 
 \al_{M, N}= \frac {\E \bigg[ 2 \Re \int_{\R}Y_N \cj{Z_M} dx-\int_{\R}|Z_M|^2  dx \bigg]}{\int_{\R^d} |\P_{\le N}f_M|^2 dx}.
\label{fmb1}
\end{align}

\noi
for $N\ge M \gg 1$.
Then, from \eqref{fM0} and \eqref{NRZ1}, we have
\begin{align}
 \al_{M, N} \sim \log M
\label{logM}
\end{align}

\noi
provided $N \gg M$ and $N$ is sufficiently large.
We are now ready to prove the divergence \eqref{pa0}.

For $M \gg 1$,
we set $f_M$, $Z_M$, and $ \al_{M, N}$ as in \eqref{fMdef}, Lemma \ref{LEM:approx}, and \eqref{fmb1}.
For the minimization problem \eqref{DPf},
we set  a drift $\dr= \dr^0$ by 
\begin{align}
\begin{split}
 \dr^0 (t) 
 & = \L^{\frac12} \bigg( -\frac{d}{dt} Z_M(t) + \sqrt{ \al_{M, N}} f_M \bigg)
\end{split}
\label{drift}
\end{align}
	
\noi
such that 
\begin{align}
\Dr^0 = I(\theta^0)(1) 
= \int_0^1 \L^{-\frac12} \dr^0(t) \, dt = - Z_M + \sqrt{ \al_{M, N}} f_M.
\label{paa0}
\end{align}

\noi
As a consequence of Lemma \ref{LEM:approx}, 
we have
\begin{lemma}
\label{LEM:bddrift}
Let $\theta^0$ as in \eqref{drift},
then we have
\[
\int_0^1 \E \big[\| \theta^0 (t)\|_{L^2 ({\R^d})}^2 \big] dt \les M^2 \log M,
\]

\noi
uniformly in $N \ge N_0(M) \gg M \ge 1$.
\end{lemma}

\begin{proof}
From \eqref{NRZ6} and \eqref{drift},
it only suffices to show that
\[
 \al_{M, N} \|  f_M \|_{\H^1({\R^d})}^2 \les M^2 \log M,
\]

\noi
which follows from \eqref{fm2} and \eqref{logM} provided $N \gg M$.
Thus we finish the proof.
\end{proof}

In what follows,
we abuse the notation and denote  
\begin{align}
 \Dr^0_N = \P_{\le N} \Dr^0
= -Z_M + \sqrt{\al_{M,N}} 
(\P_{\le N} f_M)
\label{YY0a}
\end{align}
for $N \ge M \ge 1$.
We remark that $\sqrt{\al_{M,N}} 
(\P_{\le N} f_M)$ acts as a blow-up profile in our analysis, 
and $\theta^0 \in \mathbb H_a$ is the stochastic drift such that $Y_N + \Theta^0_N$ approximates $\sqrt{\al_{M,N}} 
(\P_{\le N} f_M)$,
which drives the potential energy $R_p(Y_N + \Theta^0_N)$ to blow up.
More remarkably, 
due to the construction, 
the Wick-ordered $L^2$ norm of this approximation $Y_N + \Theta^0_N$ can be made as small as possible,
i.e. the cutoff in the Wick-ordered $L^2$ norm does not exclude the blow-up profiles. 

\begin{lemma}
\label{LEM:key}
For any $K>0$, there exists $M_0=M_0(K) \geq 1$ such that 
\begin{align}
\PP\bigg( \Big| \int_{\R^d} \wick{|Y_N (x)|^2} dx + \int_{\R^d} \big(2 \Re(Y_N \cj{\Dr^0_N}) + |\Dr^0_N|^2 \big) dx \Big| \le K \bigg) 
\ge \frac 12, 
\label{pa5}
\end{align}
	
\noi
uniformly in $N \ge N_0(M) \gg M \ge M_0$.
\end{lemma}

\begin{proof} 
First we note that the condition $N\ge N_0(M) \gg M$ guarantees that 
\[\int_{{\R^d}} |\P_{\le N}f_M|^2 dx \ges 1,\]
which further implies that $\al_{M,N} \sim \log M$.
From \eqref{paa0}, we have 
\begin{align}
\begin{split}
&\E \bigg[ \Big| \int_{\R^d} \wick{|Y_N (x)|^2} dx  + \int_{\R^d} \big(2 \Re(Y_N \cj{\Dr^0_N}) + |\Dr^0_N|^2 \big) dx \Big|^2 \bigg]\\
&= \E \bigg[ \Big|\int_{\R^d} \wick{|Y_N (x)|^2} dx  - 2 \Re \int_{{\R^d}} Y_N \cj{Z_M} dx+\int_{{\R^d}} |Z_M|^2 dx\\
&\hphantom{XX}
+   \al_{M, N} \int_{{\R^d}} |\P_{\le N}f_M|^2 dx  + 2 \sqrt{ \al_{M, N}} \Re \int_{{\R^d}} (Y_N-Z_M)f_M dx \Big|^2 \bigg].
\end{split}
\label{Pr1}
\end{align}
	
\noi
From \eqref{logM} and \eqref{NRZ5} in Lemma \ref{LEM:approx}, we have
\begin{align}
\E \bigg[ \Big| \sqrt{ \al_{M, N}}  \int_{{\R^d}} (Y_N-Z_M)f_M dx   \Big|^2   \bigg] \les M^{-\frac32} \log M.
\label{Pr4}
\end{align}	

\noi
On other hand, 
from \eqref{fmb1} and \eqref{ZZZ2}, we have	
\begin{align}
\begin{split}
 \int_{\R^d} & \wick{|Y_N (x)|^2} dx - 2 \Re \int_{{\R^d}} Y_N \cj{Z_M} dx+\int_{{\R^d}} |Z_M|^2 dx+   \al_{M, N} \int_{{\R^d}} |\P_{\le N}f_M|^2 dx \\
& =\int_{{\R^d}} |Y_N-Z_M|^2 -\E \big[|Y_N-Z_M |^2 \big] dx \\
& =\wick{\| Y_N-Z_M\|_{L^2({\R^d})\,}^2  \!}\,.
\end{split}
\label{Pr2} 
\end{align}	
	
\noi
Hence, from \eqref{Pr1}, \eqref{Pr4}, and \eqref{Pr2} with \eqref{NRZ3}
in Lemma \ref{LEM:approx}, we obtain
\begin{align*}
\E \bigg[ \Big|   \int_{\R^d} \wick{|Y_N (x)|^2} dx + \int_{\R^d} \big(2 \Re(Y_N \cj{\Dr^0_N}) + |\Dr^0_N|^2 \big) dx \Big|^2 \bigg]\les M^{- 1 }\log M.
\end{align*}

\noi
Therefore, by Chebyshev's inequality, 
given any $K > 0$, there exists $M_0 = M_0(K) \geq 1$ such that 
\begin{align*}
\PP\bigg( \Big|  \int_{\R^d} \wick{|Y_N (x)|^2} dx + \int_{\R^d} \big(2 \Re(Y_N \cj{\Dr^0_N}) + |\Dr^0_N|^2 \big) dx \Big| > K \bigg)
&\le C\frac{M^{- 1} \log M}{K^2}
< \frac 12
\end{align*}

\noi
for any $M \ge M_0 (K)$.
This proves  \eqref{pa5}.
\end{proof}

\subsection{Proof of \texorpdfstring{\eqref{pax}}{Lg}}
Let us recall that $\Theta^0 = -Z_M + \sqrt{\al_{M,N}} f_M$.
Therefore, using the mean value theorem and Young's inequality, we have for any $\eps > 0$,
\begin{align}
\begin{split} 
\big| & R_p(Y_N + \Dr^0) - R_p(\sqrt{\alpha_{M,N}}f_M)\big|\\
&\le C \int_{\R^d} |Y_N - Z_M| \big(|Y_N-Z_M| + |\sqrt{\alpha_{M,N}} f_M|\big)^{p-1}dx\\
& \le \eps R_p(\sqrt{\alpha_{M,N}}f_M) +C_\eps R_p(Y_N-Z_M).
\end{split}
\label{paa}
\end{align}
 Moreover, proceeding as in the proof of Corollary \ref{COR:intp}, 
 we have
 \begin{align}
 \begin{split}
 \E\big[R_p(Y_N -Z_M) \big] & = \frac rp \int_{\R^d} \E \bigg[ \Big|\sum_{M < n\le N}\frac{B_n(1) h_n}{\lambda_n}+\sum_{n\le M}X_n(1)h_n\Big|^p \bigg] dx\\
 &\les  \int_{\R^d} \bigg( \E \bigg[ \Big|\sum_{M < n\le N}\frac{B_n(1) h_n}{\lambda_n}+\sum_{n\le M}X_n(1)h_n\Big|^2 \bigg]\bigg)^{\frac{p}{2}}dx\\
 &\les  \int_{\R^d}\Big(\sum_{M < n\le N}\frac{1}{\lambda_n^2}h_n^2+\sum_{n\le M}\frac1{\lambda_n \sqrt M} h_n^2\Big)^{\frac{p}{2}}dx\\
 &\les  \Big(\sum_{M < n\le N}\frac{1}{\lambda_n^2}\|h_n\|_{L^p}^2+\sum_{n\le M}\frac1{\lambda_n \sqrt M} \|h_n\|_{L^p}^2\Big)^{\frac{p}{2}}\\
 &\les  \Big(\sum_{M < n\le N}\frac{1}{\lambda_n^{2+\dl}}+\sum_{n\le M} \frac1{\lambda_n^{1+\theta(p)} \sqrt M} \Big)^{\frac{p}{2}}\\
 &\les  1,
 \end{split}
 \label{paa2}
 \end{align}
where $\theta (p) >0$ provided $p< \infty$ when $d = 1,2$, and $p < \frac{2d}{d-2}$ when $d \ge 3$, uniformly in $M$.
Here we use the fact that $X_n (1)$ is a Gaussian random variable with variance
$\sim (\ld_n \sqrt M)^{-1}$.

We are now ready to put everything together.
It follows from 
\eqref{DPf},  \eqref{paa}, \eqref{paa2}, Corollary \ref{COR:intp},
and \eqref{YY0a}
that 
there exist constants $C_1, C_2, C_3 > 0 $ such that 
\begin{align*}
 -\log & \, \E_\mu\Big[\exp\Big({  R_p (u_N)} \cdot \ind_{\{|\int_{{\R^d}} : |u_N|^2 : dx|  \le K\}}\Big)   \Big]\notag \\
&\le \E\bigg[ - R_p (Y_N + \Dr^0) \cdot 
\ind_{\{ |  \int_{{\R^d}} \wick{|Y_N|^2} dx + \int_{\R^d} (2 \Re(Y_N \cj{\Dr^0_N}) + |\Dr^0_N|^2 ) dx  | \le K\}} + \frac 12 \int_0^1 \| \dr^0(t)\|_{L^2_x} ^2 dt \bigg] \notag \\
&\le  \E\bigg[ - \frac12 R_p (\sqrt{\alpha_{M,N}}f_M) \cdot \ind_{\{ |  \int_{{\R^d}} : |Y_N|^2 : dx  + \int_{\R^d} (2 \Re(Y_N \cj{\Dr^0_N}) + |\Dr^0_N|^2 ) dx | \le K\}} \notag 
  \\
 & \hphantom{XXXX} + C_\dl  R_p (Y_N - Z_M)
+ \frac 12 \int_0^1 \| \dr^0(t)\|_{L^2_x} ^2 dt \bigg] \notag \\
&\le    - \frac12 R_p (\sqrt{\alpha_{M,N}}f_M) \cdot \mathbb \PP\bigg( \Big| \int_{\R^d} \wick{|Y_N|^2} dx + \int_{\R^d} \big(2 \Re(Y_N \cj{\Dr^0_N}) + |\Dr^0_N|^2 \big) dx \Big| \le K \bigg)  \notag 
  \\
 & \hphantom{XXXX} + C_\dl \E \big[ R_p (Y_N - Z_M) \big]
+ \frac 12 \int_0^1 \E \big[ \| \dr^0(t)\|_{L^2_x} ^2 \big] dt \notag \\
\intertext{then from Lemma \ref{LEM:key}, \eqref{paa2}, and Lemma \ref{LEM:bddrift}, we may continue with} 
& \le - C_1   M^{\frac{pd}2-d}  (\log M)^\frac{p}2+ C_2  M^2 \log M  + C_3
\end{align*}

\noi
provided  $N \ge N_0(M) \gg M \ge M_0(K)$.
Therefore, it follows that
\begin{align*}
\begin{split}
 \liminf_{N \to \infty} \E_\mu\Big[& \exp\Big({  R_p (u_N)} \cdot \ind_{\{ | \int_{{\R^d}} : |u_N|^2 : dx | \le K\}}\Big)   \Big]\\
\ge &~  \exp\Big(  C_1     M^{\frac{pd}2-d}  (\log M)^\frac{p}2 - C_2  M^2 \log M   -C_3 \Big), 
\end{split}
\end{align*}
	
\noi
which diverges to infinity
as $M \to \infty$, provided that $p \ge p^*(d) = 2+\frac{4}d$.
This proves \eqref{pa0}.
 
\subsection{Proof of Theorem \ref{THM:main} - (ii) and Theorem \ref{THM:main2} - (ii)} 
Recall the decomposition \eqref{mu} 
$$ d \mu(u) = d \mu_N(u_N) \otimes d \mu_N^\perp(u_N^\perp), $$
where $u_N = \P_{\le N} u$ and $u_N^\perp = u - \P_{\le N} u$. 
Moreover, by \eqref{law}, we have that 
$$\Law(Y_N(1), Y(1) - Y_N(1)) = \mu_N \otimes \mu_N^\perp.$$
Define the set 
\begin{equation}
 \O_K^\perp = \Big\{ u_N^\perp: \Big| \int_{\R^d} \wick{|u_N^\perp|^2}dx\Big| \le \frac K2, \frac 1p \int |u_N^\perp|^p \le 1 \Big\}, \label{muperpdef}
\end{equation}
where we defined
$$ \int_{\R^d} \wick{|u_N^\perp|^2}dx = \sum_{n=N+1}^\infty \frac{|g_n|^2 - 1}{\ld_n}.$$

\noi
Notice that we need to define only the integral of the Wick power, and not the Wick power itself. This is chosen so that 
$$ \int_{\R^d} \wick{|u_N^\perp|^2}dx = \int_{\R^d} \wick{|u|^2}dx - \int_{\R^d} \wick{|u_N|^2}dx. $$
From Corollary \ref{COR:intp}, (iii) and Corollary \ref{COR:WCE}, we obtain that 
\begin{equation}\label{muperpestimate}
\mu_N^\perp(\O_K^\perp) \ge \frac12 
\end{equation}
for $N \gg1$.
From the elementary inequality 
$$ |a + b|^p \ge \frac 12 |a|^p - C |b|^p $$
for some constant $C > 0$, we deduce that 
\begin{equation}\label{elementary}
\exp\Big(\frac 1p \int_{\R^d} |u|^p dx \Big) \ge \exp\Big(-\frac Cp \int_{\R^d} |u_N^\perp|^p dx \Big) \exp\Big(\frac 1{2p} \int_{\R^d} |u_N|^p dx \Big).
\end{equation}
Therefore, by \eqref{muperpdef}, \eqref{elementary}, \eqref{muperpestimate}, and \eqref{pax} with $r = 1$, we obtain 
\begin{align*}
\mathcal Z_K &= \int \exp\Big(\frac 1p \int_{\R^d} |u|^p dx \Big) \ind_{ \{ | \int_{\R} \wick{| u (x)|^2} dx | \le K\}}   d \mu(u) \\
&= \int \exp\Big(\frac 1p \int_{\R^d} |u_N + u_N^\perp|^p dx \Big) \ind_{ \{ | \int_{\R} \wick{| u_N (x)|^2} dx  + \int_{\R} \wick{| u_N^\perp (x)|^2} dx | \le K\}} d \mu_N(u_N)d \mu_N^\perp(u_N^\perp)\\
&\ge \int \exp\Big(-\frac Cp \int_{\R^d} |u_N^\perp|^p dx \Big)\exp\Big(\frac 1{2p} \int_{\R^d} |u_N|^p dx \Big) \\
&\phantom{\int\exp\Big()} \times \ind_{ \{ | \int_{\R} \wick{| u_N (x)|^2} dx  + \int_{\R} \wick{| u_N^\perp (x)|^2} dx | \le K\}}\ind_{\O_K^\perp}(u_N^\perp) d \mu_N(u_N)d \mu_N^\perp(u_N^\perp)\\
&\ge \int e^{-C} \exp\Big(\frac 1{2p} \int_{\R^d} |u_N|^p dx \Big) \ind_{ \{ | \int_{\R} \wick{| u_N (x)|^2} dx| \le \frac K2\}}\ind_{\O_K^\perp}(u_N^\perp) d \mu_N(u_N)d \mu_N^\perp(u_N^\perp) \\
&= e^{-C}\mu_N^\perp(\O_K^\perp) \int \exp\Big(\frac 1{2p} \int_{\R^d} |u_N|^p dx \Big) \ind_{ \{ | \int_{\R} \wick{| u_N (x)|^2} dx| \le \frac K2\}}d \mu_N(u_N)\\
&= \frac{e^{-C}}2 \int \exp\Big(\frac 1{2p} \int_{\R^d} |u_N|^p dx \Big) \ind_{ \{ | \int_{\R} \wick{| u_N (x)|^2} dx| \le \frac K2\}}d \mu_N(u_N)\\
&\to \infty \text{ as } N \to \infty.
\end{align*}
This concludes the proof of \eqref{non_int2} and \eqref{non_int_d2}, and hence of Theorem \ref{THM:main} and Theorem \ref{THM:main2}.
\medskip

\begin{ackno}\rm
The authors~would like to thank Tadahiro Oh for 
helpful discussions during the preparation 
of the paper.
Y.W.~was supported by 
 the EPSRC New Investigator Award 
 (grant no.~EP/V003178/1).
\end{ackno}

\end{document}